\documentclass[12pt]{amsart} 
\usepackage[left=3cm,top=2.5cm,bottom=2.5cm,right=3cm]{geometry}
\usepackage{amssymb, amsmath,mathtools, yhmath}
\usepackage{graphicx,xspace,bbm}
\usepackage{epsfig}
\usepackage{enumitem}
\usepackage{geometry}
\usepackage[usenames,dvipsnames]{xcolor}
\usepackage{tikz}
\usepackage{gensymb}

\newcommand{\mc}[1]{\mathcal{#1}}
\usepackage[T1]{fontenc}
\usepackage[utf8]{inputenc}
\usepackage{bm,scalerel}
\usepackage[config, labelfont={normalsize}]{caption,subfig}
\newcommand{\ip}[2]{\left \langle #1, #2 \right \rangle}
\newcommand{\set}[1]{\left\{#1\right\}}
\newcommand{\norm}[1]{\left\Vert#1\right\Vert}
\newcommand{\mf}[1]{\mathbb{#1}}
\newcommand{\primes}{_E}

\newcommand{\rc}{\mathrm{Rc} }
\newcommand{\BlockE}{\text{\normalfont BlockExt} }
\newcommand{\mb}[1]{\mathbf{#1}}
\def\CummTensor{\widehat{\mathrm{R}}^\mb{x}_{\pi_f}}

\def\NB{\text{\normalfont Narc}}
\def\MaxC{\text{\normalfont MinMax}}

\def\OptTTyl{\mb{\hat{T}}^{\mb{x}}}
\def\OptT{\mb{T}^{\mb{x}}}

\def\Cumm {{\mathrm R}^\mb{x}_{\pi}}
\def\CummTensor {\hat{\mathrm{R}}^\mb{x}_{\pi}}
\usepackage{mathrsfs}
\definecolor{green}{RGB}{0,127,0}
\definecolor{red}{RGB}{191,0,0}
\usepackage[colorlinks,cite color=red,link color=green,pagebackref=true]{hyperref}
\usepackage{todonotes}
\usepackage{scalerel,stackengine}
\usetikzlibrary{matrix,arrows,calc}
\usetikzlibrary{decorations.markings}
\usetikzlibrary{patterns}
\usetikzlibrary{snakes}
\usepackage[dvipsnames]{xcolor}
\usepackage[capitalize]{cleveref}
\usepackage{cancel}
\theoremstyle{plain}
\newcommand{\Tens}{\OptTTyl} 

\def\FQ{\mathcal{F}_{\s,\q}(\HH) }
\def\T{\bar T \otimes  T }
\newtheorem{lemma}{Lemma}[section]
\newtheorem{theorem}[lemma]{Theorem}
\newtheorem{corollary}[lemma]{Corollary}
\newtheorem{proposition}[lemma]{Proposition}
\newtheorem{Prop}[lemma]{Proposition}

\newtheorem{definition}[lemma]{Definition}
\newtheorem{definition-lemma}[lemma]{Definition-Lemma}
\theoremstyle{remark}
\newtheorem{remark}[lemma]{Remark}
\newtheorem{example}{Example}
\DeclareMathOperator{\SG}{S}
\def\F{\mathcal{F}_{\rm fin}(H)}
\def\FD{\mathcal{F}_{\rm fin}(\mc{D})}

\def\Cum { R^{\mb{x},T,\lambda}(}
\newcommand{\Semi}{{\normalfont Arc} }
\def\O{\otimes}
\let\Im\undefined
\DeclareMathOperator{\Im}{\mathrm{Im}}
\allowdisplaybreaks

\numberwithin{equation}{section} 

\def\An{\text{\rm vN}(B_{\s,\q}(\HH_\R))}
\def\C{{\mathbb C}}
\def\R{{\mathbb R}}
\def\N{{\mathbb N}}

\def\D{{\rm P}}

\def\6{\, {\rm d}}
\def\i{{\rm i}}

\def\B{b_{\s,\q}}
\def\G{ B_{\alpha,q}}

\def\GZ{ B_{\alpha,q}}

\def\r{p}
\def\l{\emph{l}}
\def\r{\emph{r}}
\def\ell {n}
\def\F{\mathcal{F}_{\rm fin}(\HH)}
\def\id{I}
\def\state {\varphi} 
\def\m{\mu_{\s,\q,x,y}}

\def\P{\mathcal{P}}
\def\NC{\mathcal{NC}}

\def\Out{\rm OutArc}
\def\BL{{\mathbf{B}}}

\def\PB{\mathcal{P}^{B}}
\def\I{I}
\def\PA{\mathcal{P}^{A}}

\def\Cr{\text{\normalfont Rc}}
\def\InS{\text{\normalfont Cs}}

\def\NB{\text{\normalfont Na}}

\def\Pair{\text{\normalfont Block}}
\def\Sing{\text{\normalfont Sing}}
\renewcommand{\epsilon}{\varepsilon}

\def\H{{\bf \mathcal K}_{ n}}
\newcommand{\e}{{\epsilon}}
\def\HH{{\bf \mathcal K}}

\def\inv{{pinv}}
\def\ninv{{ninv}}
\newcommand{\q}{q}
\newcommand{\s}{\alpha}
\newcommand{\x}{\mathbf{x}}
\DeclareMathOperator{\Part}{\mathcal{P}}

\newcommand\MatchingMeandersab[3]{%
	\begin{tikzpicture}[scale=0.4]
		\foreach \x in {1,...,#1}{
			\draw[circle,fill] (\x,0)circle[radius=1mm]node[below]{};
		}
		\foreach \x/\y in {#2} {
			\pgfmathsetmacro{\Radius}{\y/2-\x/2}
			\draw(\x,0) arc[radius=\Radius, start angle=180, end angle=0];
			;
		}
		\foreach \x/\y in {#3} {
			\pgfmathsetmacro{\Radius}{\y/2-\x/2}
			\draw(\x,0) arc[radius=\Radius, start angle=-180, end angle=0];
			;}
		\foreach \x in {-#1,...,-1}{
			\draw[circle,fill] (\x,0)circle[radius=1mm]node[below]{};
		}
		\node at (-4,-0.7) { $\overline 4$}; 
		\node at (-3,-0.7) { $ \overline 3$}; 
		\node at (-2,-0.7) { $\overline  2$};
		\node at (-1,-0.7) { $\overline  1$};
		\node at (4,-0.7) { $ 4$}; 
		\node at (3,-0.7) { $  3$}; 
		\node at (2,-0.7) { $  2$};
		\node at (1,-0.7) { $  1$};
	\end{tikzpicture}
}

\newcommand\MatchingMeandersarcs[3]{%
	\begin{tikzpicture}[scale=0.7]
		\foreach \x in {1,...,#1}{
			\draw[circle,fill] (\x,0)circle[radius=1mm]node[below]{};
		}
		\foreach \x/\y in {#2} {
			\pgfmathsetmacro{\Radius}{\y/2-\x/2}
			\draw(\x,0) arc[radius=\Radius, start angle=180, end angle=0];
			;
		}
		\foreach \x/\y in {#3} {
			\pgfmathsetmacro{\Radius}{\y/2-\x/2}
			\draw(\x,0) arc[radius=\Radius, start angle=-180, end angle=0];
			;}
		\foreach \x in {-#1,...,-1}{
			\draw[circle,fill] (\x,0)circle[radius=1mm]node[below]{};
		}
		\node at (-4,-0.5) { $\overline 4$}; 
            \node at (-2,3) { $-$}; 
		\node at (-3,-0.5) { $ \overline 3$}; 
  \node at (-2.5,1) { $+$}; 
		\node at (-2,-0.5) { $\overline 2$};
		\node at (-1,-0.5) { $  \overline 1$};
  \node at (1.8,3) { $-$};
		\node at (4,-0.5) { $ 4$}; 
		\node at (3,-0.5) { $  3$}; 
		\node at (2,-0.5) { $ 2$};
  \node at (2.5,1) { $+$};
		\node at (1,-0.5) { $ 1$};
	\end{tikzpicture}}

\newcommand\MatchingMeandersabc[3]{%
\begin{minipage}{.2\textwidth}
	\begin{tikzpicture}[scale=0.6]
		\foreach \x in {1,...,#1}{
			\draw[circle,fill] (\x,0)circle[radius=1mm]node[below]{};
		}
		\foreach \x/\y in {#2} {
			\pgfmathsetmacro{\Radius}{\y/2-\x/2}
			\draw(\x,0) arc[radius=\Radius, start angle=180, end angle=0];
			;
		}
		\foreach \x/\y in {#3} {
			\pgfmathsetmacro{\Radius}{\y/2-\x/2}
			\draw(\x,0) arc[radius=\Radius, start angle=-180, end angle=0];
			;}
		\foreach \x in {-#1,...,-1}{
			\draw[circle,fill] (\x,0)circle[radius=1mm]node[below]{};
		}
		\node at (-4,-0.5){ $ \overline 4$}; 
		\node at (-3,-0.5) { $ \overline 3$}; 
		\node at (-2,-0.5) { $\overline 2$};
		\node at (-1,-0.5) { $  \overline 1$};
		\node at (4,-0.5) { $ 4$}; 
		\node at (3,-0.5) { $  3$}; 
		\node at (2,-0.5) { $ 2$};
		\node at (1,-0.5) { $ 1$};
  \node at (0,-1.5) {Block};
	\end{tikzpicture}
 \end{minipage}
}

\newcommand\MatchingMeandersexample[3]{%
	\begin{tikzpicture}[scale=0.2]
		\foreach \x in {1,...,#1}{
			\draw[circle,fill] (\x,0)circle[radius=1mm]node[below]{};
		}
		\foreach \x/\y in {#2} {
			\pgfmathsetmacro{\Radius}{\y/2-\x/2}
			\draw(\x,0) arc[radius=\Radius, start angle=180, end angle=0];
			;
		}
		\foreach \x/\y in {#3} {
			\pgfmathsetmacro{\Radius}{\y/2-\x/2}
			\draw(\x,0) arc[radius=\Radius, start angle=-180, end angle=0];
			;}
		\foreach \x in {-#1,...,-1}{
			\draw[circle,fill] (\x,0)circle[radius=1mm]node[below]{};
		}
	\end{tikzpicture}
}

\newcommand\MatchingMeandersPositiveTracea[3]{%
	\begin{tikzpicture}[scale=0.35]
		\foreach \x in {1,...,#1}{
			\draw[circle,fill] (\x,0)circle[radius=1mm]node[below]{};
		}
		\foreach \x/\y in {#2} {
			\pgfmathsetmacro{\Radius}{\y/2-\x/2}
			\draw(\x,0) arc[radius=\Radius, start angle=180, end angle=0];
			;
		}
		\foreach \x/\y in {#3} {
			\pgfmathsetmacro{\Radius}{\y/2-\x/2}
			\draw(\x,0) arc[radius=\Radius, start angle=-180, end angle=0];
			;}
		\foreach \x in {-#1,...,-1}{
			\draw[circle,fill] (\x,0)circle[radius=1mm]node[below]{};
		}
		\node at (-3.8,-0.6) { $\scriptscriptstyle{ \bar{s}_{i}}$};
        \node at (3.8,-0.6) { $\scriptscriptstyle{ s_{i}}$};
		\node at (-0.8,-0.6) { $\scriptscriptstyle{\bar 1}$};
		
		\node at (8,-0.6) { $\scriptscriptstyle{ s_r}$}; 
         \node at (-8,-0.6) { $\scriptscriptstyle{ \bar{s}_{r}}$}; 
		\node at (1.2,-0.6) { $\scriptscriptstyle{ 1}$}; 
		\node at (-2.5,4.5) { $\scriptscriptstyle{ \overline{W}_j}$}; 
		\node at (2.5,4.5) { $\scriptscriptstyle{ {W}_j}$}; 
		\node at (9,4) { $ \mapsto$};
	\end{tikzpicture}
}

\newcommand\MatchingMeandersPositiveTraceb[3]{%
	\begin{tikzpicture}[scale=0.35]
		\foreach \x in {1,...,#1}{
			\draw[circle,fill] (\x,0)circle[radius=1mm]node[below]{};
		}
		\foreach \x/\y in {#2} {
			\pgfmathsetmacro{\Radius}{\y/2-\x/2}
			\draw(\x,0) arc[radius=\Radius, start angle=180, end angle=0];
			;
		}
		\foreach \x/\y in {#3} {
			\pgfmathsetmacro{\Radius}{\y/2-\x/2}
			\draw(\x,0) arc[radius=\Radius, start angle=-180, end angle=0];
			;}
		\foreach \x in {-#1,...,-1}{
			\draw[circle,fill] (\x,0)circle[radius=1mm]node[below]{};
		}
	 
		\node at (-3.8,-0.6) { $\scriptscriptstyle{ \overline s_{i}}$}; 
		\node at (-0.8,-0.6) { $\scriptscriptstyle{\bar 1}$};
		 
		\node at (4.2,-0.6) { $ \scriptscriptstyle{ s_{i}}$}; 
		\node at (1.2,-0.6) { $\scriptscriptstyle{ 1}$};
		
		\node at (-9.2,-0.6) { $\scriptscriptstyle{\overline{n}}$};  
        \node at (9.2,-0.6) { $\scriptscriptstyle{n}$};  
        \node at (8,-0.6) { $\scriptscriptstyle{s_{r}}$};
        \node at (-8,-0.6) { $\scriptscriptstyle{\bar{s}_{r}}$};
		\node at (-2.5,4.5) { $\scriptscriptstyle{ \overline{W}_j}$}; 
		\node at (2.5,4.5) { $\scriptscriptstyle{ {W}_j}$}; 
	\end{tikzpicture}
	
}

\newcommand\MatchingMeandersnew[2]{%
	\begin{tikzpicture}[scale=0.7]
		\foreach \x in {1,...,#1}{
			\draw[circle,fill] (\x,0)circle[radius=1mm]node[below]{$ \x$};
		}
		\foreach \x/\y in {#2} {
			\pgfmathsetmacro{\Radius}{\y/2-\x/2}
			\draw(\x,0) arc[radius=\Radius, start angle=180, end angle=0];
			;
			\node at (0,5.45) { $\scriptstyle{\blacklozenge}$}; 
			\node at (0,1.7) { $\scriptstyle{\blacklozenge}$}; 
			\draw[] (0,0) -- (0,6);
		}
		
		\foreach \x in  {#1,...,1}{
			\draw[circle,fill] (-\x,0)circle[radius=1mm]node[below]{$ \overline{\x}$};
		}
	\end{tikzpicture}
}

\newcommand\MatchingMeanderscol[2]{%
	\begin{tikzpicture}[scale=0.7]
		\foreach \x in {1,...,#1}{
			\draw[circle,fill] (\x,0)circle[radius=1mm]node[below]{$ \x$};
		}
		\foreach \x/\y in {#2} {
			\pgfmathsetmacro{\Radius}{\y/2-\x/2}
			\draw(\x,0) arc[radius=\Radius, start angle=180, end angle=0];
			;
			\node at (0,5.45) { $\scriptstyle{\blacklozenge}$}; 
			\node at (0,4.2) { $\scriptstyle{\blacklozenge}$}; 
			\node at (0,1.7) { $\scriptstyle{\blacklozenge}$}; 
			\draw[] (0,0) -- (0,6);
		}
		
		\foreach \x in  {#1,...,1}{
			\draw[circle,fill] (-\x,0)circle[radius=1mm]node[below]{$ \overline{\x}$};
		}
    \draw[circle,fill=yellow](-8.8,1.5)circle[radius=1mm];
    \draw[circle,fill=yellow](-2.1,5.3)circle[radius=1mm];
     \draw[circle,fill=yellow](-1.5,5.1)circle[radius=1mm];
     \draw[circle,fill=yellow](-1.7,1.9)circle[radius=1mm];
     \draw[circle,fill=yellow](-5.83,1.45)circle[radius=1mm];
     \draw[circle,fill=yellow](-4.83,1.4)circle[radius=1mm];

	\end{tikzpicture}
}

\begin{document}
	\title{The Poisson type operators on the Double Fock Space of Type B}

	\author[Wiktor Ejsmont]{Wiktor Ejsmont}
	\address[Wiktor Ejsmont]{ Department of Telecommunications and Teleinformatics, 
		Wrocław University of Science and Technology \\ Wybrze\.ze Wyspia{\'n}skiego 27, 50-370 Wroc\l aw, Poland }
	\email{wiktor.ejsmont@gmail.com}

  \author[Patrycja H\k{e}{\'c}ka-J\k{e}draszczyk]{Patrycja H\k{e}{\'c}ka-J\k{e}draszczyk}
 \address [Patrycja H\k{e}{\'c}ka-J\k{e}draszczyk]{Department of Telecommunications and Teleinformatics, Wroclaw University
of Science and Technology\\
Wybrze\.ze Wyspia\'nskiego 27, 50-370 Wroc\l aw, Poland}
\email{patrycja.hecka@gmail.com}
	\maketitle
	
	\begin{abstract}
		 The double Fock space of type B was introduced in 2023 by Bożejko and Ejsmont (\cite{BE23}). In this article, we show the acting of Poisson type operators in that space. For this purpose, we define the double gauge operators (analogous to \cite{Ans01}, \cite{Ejsmont1}) and compute the multidimensional moments of a joint distribution of Poisson operators. We show that the presented method of calculating negative arcs and restricted crossings is compatible with counting positive and negative inversions on a Coxeter group of type B. The present method is much simpler than using colored type-B set partitions in the sense of \cite{Ejsmont1}.
	\end{abstract}

	\section{Introduction}
 Field operators, such as sums of creation and annihilation operators, on symmetric Fock space give rise to Gaussian distributions in the vacuum state. Hudson and Parthasarathy \cite{HudPar} and Sch{\"u}rmann \cite{SchurCondPos} observed that by adding an appropriate
gauge component, we can obtain an operator with Poisson distribution. Bożejko and Speicher used the Coxeter groups of type A  to construct a q-deformed Fock space and a q-deformed Brownian motion \cite{BozejkoSpeicher1991} (= Fock
space of type A). Bożejko, Ejsmont and Hasebe followed this idea in \cite{BEH15} and constructed an $(\alpha,q)$-Fock space using the Coxeter groups of type B. In \cite{Ejsmont1} Ejsmont presented acting of the Poisson type operators on the Fock space of type B. In \cite{BE23} Bożejko and Ejsmont gave an alternative construction to \cite{BEH15} of a generalized Gaussian process related to Coxeter groups of type B. They constructed a deformed probability space by using some symmetrization operator and obtained the result that the joint moments of a Gaussian operator may be expressed by the analogue of
statistic on the set of pair partitions. The proposed construction significantly simplifies combinatorics, which otherwise aligns with counting inversions in Coxeter groups. They introduced a double Fock space of type B.  The main aim of this work is to present how the Poisson type operators act on the double Fock space of type B.

Firstly we recall the work of  Bo\.zejko and Speicher about $q$-Gaussian process \cite{BozejkoSpeicher1991}  on the $q$-deformed Fock space 
$
\mathcal{F}_q(H):=(\mathbb{C}\Omega)\oplus\bigoplus_{n=1}^\infty H^{\otimes n} $, where $-1\leq  q \leq 1$, $\Omega$ denotes the vacuum vector and $H$ is the complexification of some real separable
Hilbert space $H_\R$.  On this space  the authors introduced 
  a deformed inner product, using the following  symmetrization:
\begin{align} 
\sum_{\sigma\in \SG(n)}q^{inv(\sigma)}\sigma, 
\end{align}
where   $\SG(n)$  is the set of all the permutations of $\{1,\dots,n\}$ (the Coxeter groups of type A) and $inv(\sigma):=\mbox{card}\{(i,j):i<j,
\sigma(i)>\sigma(j)\}$ is the number of inversions of  $\sigma\in
\SG(n)$. 

In this article on the double Fock space of type B we define a double creation operator $b_{\alpha,q }^\ast(x \otimes y)$, its adjoint, that is, the double annihilation operator $b_{\alpha,q }(x \otimes y)$, a gauge operator $p(\bar T \otimes T)$ (analogous to Anshelevich \cite{Ans01}, see also \cite{Ejsmont2020}) and a Poisson operator $B^{\lambda_1,\lambda_2}_{\alpha,q}(x \otimes y):=b_{\alpha,q }^\ast(x \otimes y)+b_{\alpha,q }(x \otimes y)+p(\bar T \otimes T)+\lambda_1 \otimes \lambda_2$, $\lambda_1,\lambda_2 \in \mathbb{R}, x \otimes y \in H_\R\otimes {H}_\R$.
The key point  is  computing multidimensional moments of the distribution, which are  given by:
\begin{equation}
\begin{split}
\state(\G(x_{\overline{ n}}\otimes x_{ n})\cdots \G(x_{\bar 1}\otimes x_{ 1}))= \sum_{\pi\in\PB(n)}  \s^{\NB(\pi)}\q^{\Cr(\pi)}  R_{\pi}^{x ,T,\lambda}.
\end{split}
\label{eq:multimomentintro}
\end{equation}
where $\text{Rc}(\pi)$ is the number of crossings of a partition $\pi$ (see \cite{Bia97}), $\text{Na}(\pi)$ is the number of negative arcs (see \cref{def:arcs}), $R_{\pi}^{x ,T,\lambda}$ is B cumulant, (see  \cref{def: bcumulant}) and $\PB(n)$  is the set of partitions of type B (see \cref{def:partycji}). 

	\section{Preliminaries }
	\label{sec:preliminaries}
	In the present section we recall that the construction of type B Fock space \cite{BEH15} can be adapted to recover the
double Fock space of type B \cite{BE23}.
	In the following we will briefly describe the tools we use in this investigation, 
	which is partially contained in the previous  articles \cite{BEH15}, \cite{BE23}. 
	For further information, the reader is referred to \cite{BEH15} and \cite{BE23} and the references therein. 
	\subsection{Coxeter groups of type B}
	The Coxeter
	group of type B (hyperoctahedral group) of degree $n$, denoted by $B(n)$, is generated
	by the elements $\pi_0, \pi_1, \dots , \pi_{n-1}$ satisfy the defining relations $\pi_i^2=e, 0\leq i \leq n-1$,  $(\pi_0\pi_1)^4=(\pi_i \pi_{i+1})^3=e, 1\leq i < n-1$ and $(\pi_i \pi_j)^2=e$ if $|i-j|\geq2, 0\leq i,j\leq n-1$.  Recall that $\{\pi_i\mid i=1,\dots,n-1\}$ generates the symmetric group $S(n)$. 
	The Coxeter diagram for $B(n)$  is described in Figure \ref{fig:BN}.
	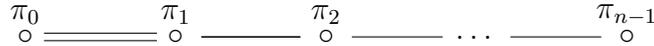
\begin{figure}[htp]
		\begin{center}
			\begin{tikzpicture} 
				[scale=.5,auto=left,every node/.style={circle}]
				\node (n7) at (1,3.6) {$\pi_{n-1}$};  
				\node (n6) at (1,3) {$\circ$};
				\node (n5) at (-3,3)  {$\dots$};
				\node (n3) at (-7,3.6)  {$\pi_2$};
				\node (n1) at (-7,3)  {$\circ$};
				\node (n4) at (-11,3.6)  {$\pi_{1}$};
				\node (n2) at (-11,3)  {$\circ$};
				\node (n8) at (-15,3.6)  {$\pi_{0}$};
				\node (n0) at (-15,3)  {$\circ$};
				
				\foreach \from/\to in {n2/n1, n1/n2,n1/n5,n5/n6}
				\draw (\from) -- (\to);
				\draw  (-14.5,3.1) -- (-11.5,3.1);   
				\draw  (-14.5,2.9) -- (-11.5,2.9);      
			\end{tikzpicture}
			\caption{The Coxeter diagram for $B(n).$}
			\label{fig:BN}
		\end{center}
	\end{figure}
	

	\noindent
	We define $\sigma \in B(n)$ in an irreducible form
	$$
	\sigma=\pi_{i_1} \cdots \pi_{i_{k}}, \qquad 0\leq i_1,\dots,i_k \leq n-1,
	$$
	i.e., in a form with the minimal length, and in this case let 
	\begin{align*}
		&l_1(\sigma)= \text{the number of the occurrences of the factor $\pi_0$  in $\sigma$}, \\
		&l_2(\sigma) =  \text{the number of the occurrences of all the factor of the form $\pi_1,\dots,\pi_{n-1}$ in $\sigma$}. 
	\end{align*}
	\begin{remark}
		These definitions do not depend on the way we express $\sigma$ in an irreducible form, and therefore $l_1(\sigma)$ and $l_2(\sigma)$ are well defined (see \cite[Proposition 1]{BozejkoSzwarc2003}). 
	\end{remark}

\begin{remark}

The length functions $l_1$ and $l_2$ are very related to the root system
of the Coxeter group of type B.
We can use the results of \cite[Proposition 1]{BozejkoSzwarc2003}
or \cite[Chapter
4.3]{Bourbaki}.
In our case, the group $B(n)$ is related to the root system of
type B  $$\Pi=\Pi_1\cup \Pi_2,$$
where  $\Pi_1=\{e_1,\dots,e_n\}$ and  $\Pi_2=\{e_i\pm e_j\mid i<j\}$. 
The related  length functions (see \cite[Proposition 1]{BozejkoSzwarc2003}) in our group $B(n)$ are the following
$$l_i (\sigma) = \#\{ \Pi_i  \cap \sigma^{-1} ( -  \Pi _i )\}, \quad i=1,2, $$
which
in our notations reads 
$R_2^{+}=\Pi_{2}$ and  $R_1^{+}    =  \Pi_{1}$.
More information on the subject can be found in the books \cite{BjornerBrenti,Humphreys}
\end{remark}

	\subsection{Orthogonal polynomials}
	In this subsection we remind basic facts about the
	orthogonal polynomials.  
	
	For a probability measure $\mu$ with finite moments of all orders, we can assign orthogonal polynomials $(P_n(x))_{n=0}^\infty$  
	with $\text{deg}\, P_n(x) =n$ 
	and the leading coefficient of each  $P_n(x)$ is $1$ i.e., monic. 
	It is known that they satisfy the recurrence relation
	\begin{align} \label{wielortogonalnerekursia}
		x P_n(x) = P_{n+1}(x) +\beta_n P_n(x) + \gamma_{n-1} P_{n-1}(x),\qquad n =0,1,2,\dots
	\end{align}
	with the convention that $P_{-1}(x)=0$. The coefficients $\beta_n$ and $\gamma_n$ are called \emph{Jacobi parameters} and satisfy $\beta_n \in \R$ and $\gamma_n \geq 0$. The continued fraction representation of the Cauchy transform can be expressed in terms of the Jacobi Parameters:
\[
\int_{\R}\frac{\mu(d t)}{z-t} = \dfrac{1}{z-\beta_0 -\dfrac{\gamma_0}{z-\beta_1-\dfrac{\gamma_1}{z- \beta_2 - \cdots}}}.
\]
	It is known that 
	\begin{equation}\label{eq54}
		\gamma_0 \cdots \gamma_n=\int_{\R}|P_{n+1}(x)|^2\mu(d x),\qquad n \geq 0.
	\end{equation} 
	Let 
	\begin{itemize}
		\item $[n]_q$ be the $q$-number 
		$
		[n]_\q:= 1+\q+\cdots+\q^{n-1},\qquad n \geq1
		$;
		\item  $(s;\q)_n$ be the $\q$-Pochhammer symbol
		$
		(s;q)_n:= \prod_{k=1}^n(1-s \q^{k-1}),\text{ }s \in \R, |\q|<1, n \geq1. 
		$
	\end{itemize}

	$(\alpha,q)$-Poisson of type B polynomials are defined by the recursion relations
	\begin{equation}\label{recursion}
		t Q_n^{(\s, \q)}(t) = Q_{n+1}^{(\s, \q)}(t) +[n]_q(1 + \s \q^{n-1})Q_{n-1}^{(\s, \q)}(t)+[n]_q(1 + \s \q^{n-1})Q_{n}^{(\s, \q)}(t), \qquad n=0,1,2,\dots,
	\end{equation}
	where $Q_{-1}^{{\s,\q}}(t)=0,Q_0^{{\s,\q}}(t)=1$ and $-1 \leq \s,\q \leq 1$.  There exists a probability measure $\mu_{\alpha,q}$ which is associated with orthogonal polynomials $Q_{ n}^{(\alpha,q)}$.
	\begin{remark}
	 \label{uwagimiara}   
	
	 The measure of orthogonality of the above polynomial sequence is not known (we were informed about this by professor Mourad E. H. Ismail). In special cases, we can identify this measure:
\begin{enumerate}
\item the measure $\mu_{\alpha,1}$ is the classical Poisson law; 
\item the measure $\mu_{0,0}$ is the Marchenko-Pastur distribution; 
\item the measure $\mu_{0,q}$ is the $q$-Poisson law and the orthogonal polynomials $Q_n^{(0,q)}(t)$ are called \emph{$q$-Poisson-Charlier polynomials} (see \cite{Ans01}); 
\item the measure $\mu_{\alpha,-1}$ is a non-symmetric Bernoulli distribution; 
\item the measure $\mu_{\alpha,0}$,   $\alpha\neq 0$ and $\alpha>-1$ is a two-state free Meixner distribution because its Jacobi parameters are independent of $n$ for $n\geq 2$ (see \cite{A03}). 
This corresponds to a probability measure $\mu_{\alpha,0}$ on $\R$ which is given by  
$$\begin{cases} \frac{\sqrt{4-(x-1)^2}}{p_\alpha(x)}{\bf 1}_{(-1,3)}(x)dx &\text{for  } \alpha \in (-1,0]\\ \frac{\sqrt{4-(x-1)^2}}{p_\alpha(x)}{\bf 1}_{(-1,3)}(x) dx+\lambda_\alpha\delta_{x_\alpha} &\text{for } \alpha \in (0,\infty) \end{cases}
,$$
where $p_\alpha(x)=\frac{2\pi}{\alpha+1} \left(-x^3\alpha+x^2\alpha^2+x(2\alpha^2+3\alpha+1)+(\alpha+1)^2\right)$ and 
\begin{align*}
    \lambda_\alpha&=\frac{\alpha^2 + \sqrt{2}\alpha\sqrt{\frac{\alpha^3 + (\alpha^2-1)\sqrt{\alpha(\alpha + 4)} + 2\alpha^2 - 3\alpha + 2}{\alpha}} + (\alpha-1)\sqrt{\alpha(\alpha + 4)} + \alpha}{2\alpha(\alpha^2 + (\alpha+3)\sqrt{\alpha(\alpha + 4)} + 5\alpha + 4)}
   \\ x_\alpha&=(\alpha+1)\frac{\alpha+\sqrt{\alpha(\alpha+4)}}{2\alpha},\qquad \alpha \in (0,\infty).
    \end{align*}
\end{enumerate}
\end{remark}
\begin{proof}[Proof of point (5)]
The corresponding  system of orthogonal polynomial has the form 
\begin{align*}
	t Q_n^{(\alpha, 0)}(t) &= Q_{n+1}^{(\alpha, 0)}(t) +(1+\alpha)Q_{n-1}^{(\alpha, 0)}(t)+(1+\alpha)Q_{n}^{(\alpha, 0)}(t) , \qquad  n=1;
	\\ t Q_n^{(\alpha, 0)}(t) &= Q_{n+1}^{(\alpha, 0)}(t) +Q_{n-1}^{(\alpha, 0)}(t)+Q_{n}^{(\alpha, 0)}(t), \qquad  n=2,3,\dots.
\end{align*} 
where  $Q_{-1}^{(\alpha, 0)}(t)=0,Q_0^{(\alpha, 0)}(t)=1, Q_1^{(\alpha, 0)}(t)=t,$ and $\alpha\in (-1,\infty)$. 
 The continued fraction representation of the Cauchy transform can be expressed in terms of the Jacobi Parameters:
\begin{align*}G_{\mu_{\alpha,0}}(z)&=
\int_{\R}\frac{\mu(d t)}{z-t} \\&= \dfrac{1}{z -\dfrac{1+\alpha}{z-(1+\alpha)-\dfrac{1}{z- 1 - \dfrac{1}{z- 1 - \cdots}}}}
\\&= \dfrac{1}{z -\dfrac{1+\alpha}{z-(1+\alpha)-G_\rho(z)}}
\intertext{where $\rho$ is the semicircular distribution with mean  and variance one. We now expand further and obtain}&= \dfrac{1}{z -\dfrac{1+\alpha}{z-(1+\alpha)-\frac
{z-1-\sqrt{(z-1)^2-4}}{2}}}
\\&=\frac{ z-(1+2\alpha)+\sqrt{(z-1)^2-4}}{z^2-z(1+2\alpha)-2(1+\alpha)+z\sqrt{(z-1)^2-4}}
\end{align*}
where the branch of the analytic square root should be determined by the condition
that $\Im (z) > 0\implies \Im G_{\mu_{\alpha,0}}(z) < 0$. We  choose a branch of $\sqrt{(z-1)^2-4}$ for $z$ in the
upper half-plane $\C^+$. We write $\sqrt{(z-1)^2-4}=\sqrt{z-3}\sqrt{z+1}$  and define each of $\sqrt{z-3}$
and $\sqrt{z+1}$ in $\C^+$. 
For  $z\in \C^+$, let $\theta_1$ be the angle between the x-axis and the line
joining $z$ to $3$; and $\theta_2$  the angle between the x-axis and the line joining $z$ to $-1$ that is $0<\theta_1;\theta_2 <\pi$. Then we define $\sqrt{z-3}\sqrt{z+1}$ to
be $\sqrt{|z-3|}\sqrt{|z+1|}e^{i(\theta_1+\theta_2)/2}=\sqrt{|(z-1)^2-4|}e^{i(\theta_1+\theta_2)/2}$.

This representation is useful to calculate the measure. 
We have 
\begin{align*}
 \lim_{\epsilon \to 0^+}\Im  \sqrt{(x+i\epsilon-1)^2-4}&=\lim_{\epsilon \to 0^+}\Im\sqrt{|(x+i\epsilon-1)^2-4|}\sin((\theta_1+\theta_2)/2)\\
 &=\lim_{\epsilon \to 0^+}\begin{cases}\Im  \sqrt{|(x-1)^2-4|} \times e^{i 0} &\text{for  } x \notin (-1,3) \\ \Im  \sqrt{|(x-1)^2-4|} \times e^{i\pi/2}    &\text{for } x \in (-1,3) \end{cases}\\&=\begin{cases} \sqrt{|(x-1)^2-4|} \times 0 &\text{for  } x \notin (-1,3) \\  \sqrt{|(x-1)^2-4|} \times 1=\sqrt{4-(x-1)^2}    &\text{for } x \in (-1,3). \end{cases}   
\end{align*}

The Cauchy
transform uniquely determines the measure, and there is an inversion formula called Stieltjes
inversion formula.  Let $x_\epsilon=x+i\epsilon$ and $f(z)=\sqrt{(z-1)^2-4},$  
then 
\begin{align*}
d\mu(x)&=-\frac{1}{\pi}\lim_{\epsilon \to 0^+}\Im G_{\mu}(x+i\epsilon)\\&=-\frac{1}{\pi}\lim_{\epsilon \to
  0^+}\Im \frac{ \overbracket{x_\epsilon-(1+2\alpha)}^a+f(x_\epsilon)}{x_\epsilon^2-x_\epsilon(1+2\alpha)-2(1+\alpha)+x_\epsilon f(x_\epsilon)} \frac{\overbracket{x_\epsilon^2-x_\epsilon(1+2\alpha)-2(1+\alpha)}^b- {x_\epsilon f(x_\epsilon)}}{x_\epsilon^2-x_\epsilon(1+2\alpha)-2(1+\alpha)- {x_\epsilon f(x_\epsilon)}}
  \\&=-\frac{1}{\pi}\lim_{\epsilon \to
  0^+}\Im \frac{ -a{x_\epsilon f(x_\epsilon)}+bf(x_\epsilon)+ab-{x_\epsilon }f(x_\epsilon)^2}{(x_\epsilon^2-x_\epsilon(1+2\alpha)-2(1+\alpha))^2-x_\epsilon^2 f^2(x_\epsilon)} 
\\&=-\frac{1}{\pi} \frac{-(\alpha+1)\sqrt{4-(x-1)^2}}{2\left(-x^3\alpha+x^2\alpha^2+(2\alpha^2+3\alpha+1)x+\alpha^2+2\alpha+1\right)} \quad\text{ for }x\in(-1,3).
\end{align*}

Now we shall see the discrete part of the measure $\mu_{\alpha,0}$. 
To determine the atoms, we should compute the  limits 
$
\lim_{\epsilon \to 0^+}i\epsilon G_\mu\left({x+i\epsilon}\right),
$  where we have to take the suitable branch of the square root. 
Taking care of the choice of the branches of
the analytic square root appeared in $G_{\mu_{\alpha,0}}(z) $, it follows  that considered limit is potentially nonzero  at real roots of
\begin{align*}
0&=z^2-z(1+2\alpha)-2(1+\alpha)+z\sqrt{(z-1)^2-4}
 \intertext{we use the branch of $\sqrt{(z-1)^2-4}$ defined before and obtain }
    &=\begin{cases} z^2-z(1+2\alpha)-2(1+\alpha)+z\sqrt{|(z-1)^2-4|}e^{i0}&\text{for  } z \geq 3;  \\  z^2-z(1+2\alpha)-2(1+\alpha)+z\sqrt{|(z-1)^2-4|}e^{i{\pi}/{2}}    &\text{for } z \in [-1,3);  
    \\   z^2-z(1+2\alpha)-2(1+\alpha)+z\sqrt{|(z-1)^2-4|}e^{i{\pi}}    &\text{for } z< -1 .  
    \end{cases} 
   \quad \begin{matrix}
$(I)$ &\\
$(II)$ & \\ $(III)$&
\end{matrix}
\end{align*}
The first equation $(I)$ can be rewritten as
\begin{align*}
(z+1)(-\alpha z^2+z(\alpha^2+\alpha)+\alpha^2+2\alpha+1)=0
\end{align*}
The equation $-\alpha x^2+x(\alpha^2+\alpha)+\alpha^2+2\alpha+1=0$ has two real solutions  for  $\alpha >0$ 
$$x_\alpha=(\alpha+1)\frac{\alpha+\sqrt{\alpha(\alpha+4)}}{2\alpha}\geq 3, \quad 
x'_\alpha=(\alpha+1)\frac{\alpha-\sqrt{\alpha(\alpha+4)}}{2\alpha}<3.
$$ 
The  roots $-1$ and $x'_\alpha$ are discarded because the appropriate branch of the square root required $z\geq 3.$   On the other hand the second  equation $(II)$
$$z^2-z(1+2\alpha)-2(1+\alpha)+iz\sqrt{(z-1)^2-4}=0$$
has one real roots $z=-1$. 
Finally, the last equation $(III)$ has no solution. In fact, the left-hand side of the  equation
$$z^2-z(1+2\alpha)-2(1+\alpha)=-z\sqrt{|(z-1)^2-4|}e^{i{\pi}}=z\sqrt{|(z-1)^2-4|},$$
is positive and the right-hand side is negative under $z<-1$ and $\alpha>-1$.
For $z= -1$ we have
$
\lim_{\epsilon \to 0^+}i\epsilon G_\mu\left({-1+i\epsilon}\right)=0,  
$ which can be trivially obtained by simple direct calculation. 
For $x_\alpha$ we have $\lim_{\epsilon \to 0^+}i\epsilon G_\mu\left({x_\alpha+i\epsilon}\right)=\lambda_\alpha.$ 
\end{proof}

	\section{The double Fock space  of type B}
	\label{subsec:Fock}
	 In this subsection we recall the construction of a double Fock space of type B, which is partially investigated in \cite{BE23}. 
	Let $H_\R$ be a separable real Hilbert space and let $H$ be its
	complexification with the inner product $\langle\cdot,\cdot\rangle,$ linear in the right component and anti-linear in the left. 
	A Hilbert space $\HH:=H\otimes {H}$ is the complexification of its
	real subspace $\HH_\R:=H_\R\otimes {H}_\R$, with the inner product
	\[\langle x\otimes y,\xi\otimes \eta  \rangle_{\HH} = \langle
	x,\xi\rangle\langle y,\eta\rangle.\]

	We define $\H:=H^{\otimes n} \otimes {{H}^{\otimes n}}$
 for $n\in \N$ and 
let $\F$ be the algebraic full Fock space over $\HH:$
	\begin{equation}
		\F:= \bigoplus_{n=0}^\infty \H= \bigoplus_{n=0}^\infty  H^{\otimes 2n}, 
	\end{equation} 
	with the convention that $\HH^{\otimes 0} =H^{\otimes 0} \otimes H^{\otimes 0}=\C\Omega \otimes \Omega$ is the one-dimensional normed space along with the unit vector $\Omega \otimes \Omega$ and free inner product $\langle\cdot,\cdot \rangle_{0,0}$. 
	We equip $\F$ with the inner product  
	\begin{align*}
	&\langle x_{\overline n} \otimes \cdots \otimes x_n, y_{\overline m} \otimes \cdots  \otimes y_m\rangle_{\alpha,q} := \delta_{m,n}\langle x_{\overline n} \otimes \cdots \otimes x_n,P_{\s,\q}^{(n)}  y_{\overline n}\otimes\cdots  \otimes y_{n}\rangle_{0,0}    
\end{align*}
where $P_{\s,\q}^{(n)}$ is the \emph{type B symmetrization operator} acting on the algebraic full Fock space
	\begin{align*}	&P_{\s,\q}^{(n)}:= \sum_{\sigma \in B(n)}\s^{l_1(\sigma)} \q^{l_2(\sigma)} \, \sigma,\qquad n \geq1, \\ 
		&P_{\s,\q}^{(0)}:= \id_{H^{\otimes 0}\otimes H^{\otimes 0}}. 
	\end{align*}  
and for  $\sigma \in B(n)$ we define  
 \begin{align*}
		\sigma  :\H&\to \H \\
		y_{\overline n}\otimes\cdots \otimes y_n &\mapsto y_{\sigma^{-1}(\overline n)}\otimes\cdots \otimes y_{\sigma^{-1}(n)}.
	\end{align*}


We use the notation $x_{\overline n}\otimes \cdots  \otimes x_{n}$  for the elements  of $\H$
namely
$$x_{\overline n}\otimes \cdots  \otimes x_{n}=x_{\overline n}\otimes\cdots \otimes x_{\overline 1} \otimes  x_1\otimes\cdots \otimes x_n.$$

	 \begin{remark}
	     Bożejko and Ejsmont \cite{BE23} showed that if $|\q|\leq 1$, $|\s|\leq 1$ and $(\s,\q)\notin\{(\pm 1,\pm 1),(0,\pm 1),(\pm1,0)\}$,  then considered  symmetrization  is strictly positive, $\ker {P_{\s,\q}^{(n)}}=\{0\},$ i.e.,
$$ \langle \xi,{P_{\s,\q}^{(n)}} \xi \rangle_{\s,\q} \geq c_n  \|\xi\|^2\quad \text{for } \xi \in \H \text{ and for some positive constant } c_n>0.$$ They also proved that operator $P^{(n)}_{\s,\q}$ has a kernel if and only if $(\s,\q)\in\{(\pm 1,\pm 1),(0,\pm 1),(\pm1,0)\}.$
	 \end{remark}

 Let $\q,\s \in (-1,1)$. The algebraic full Fock space $\F$ equipped with the inner product $\langle\cdot,\cdot \rangle_{\s,\q}$ is called the \emph{double Fock space of type B} and denoted by $\mathcal{F}_{\s,\q}(\HH)$. 
	
	For $x\in H$ let $l(x)$ and $r(x)$ be the free left and right
	annihilation and let $l^\ast(x)$ and $r^\ast(x)$ be the free left and right
	creation  operators on $H^{\otimes n} $,  respectively, defined by
	the equations
	\begin{align*}
		l^\ast(x)(x_1\otimes \dots \otimes x_n )&:=x\otimes x_1\otimes \dots \otimes x_n  ,
		\\
		l(x)(x_1\otimes \dots \otimes x_n )&:=\langle  x, x_1 \rangle  x_2\otimes \dots \otimes x_n , \\
		r^\ast(x)(x_1\otimes \dots \otimes x_n )&:= x_1\otimes \dots \otimes x_n \otimes x,\\
		r(x)(x_1\otimes \dots \otimes x_n )&:=\langle  x, x_n \rangle  x_1\otimes \dots \otimes x_{n-1}, 
		\intertext{where the adjoint is taken with respect to the free inner product. 
			The  left-right creation and  annihilation operators $\r^\ast(x\otimes y), \r(x\otimes y)$ on $\F$  are defined as follows: }
		\r^\ast(x\otimes y)(x_{\bar n}\otimes \dots \otimes  x_n)&:=l^\ast(x) 
		r^\ast( y) x_{\bar n}\otimes \dots \otimes x_n, \quad  &&\r^\ast(x\otimes y)\Omega \otimes \Omega:=x \otimes y, \\
		\r(x\otimes y)(x_{\bar n}\otimes \dots \otimes x_n )&:= l( x) r(y)x_{\bar n}\otimes \dots \otimes x_n  , \quad &&\r(x\otimes y)\Omega \otimes \Omega :=0,
	\end{align*}
	where $x_{\bar n}\otimes\dots \otimes x_n\in  \H,\text{ }
	n\geq 1$. 
	It holds that $[\r^\ast(x\otimes y)]^\ast = \r(x\otimes y)$ where the adjoint is taken with respect to $\langle \cdot, \cdot \rangle_{0,0}$. 

 The following proposition  \cite[Proposition 3.5]{BE23}  is useful to  calculate the annihilator and Poisson type operator.

	\begin{proposition}\label{prop1}
		We have the decomposition 
		\begin{equation}\label{decomposition}
			P^{(n)}_{\s,\q}=\left( \id \otimes P^{(n-1)}_{\s,\q}\otimes \id \right)R^{(n)}_{\s,\q}={R^{(n)}_{\alpha,q}}^*(I\otimes P^{(n-1)}_{\alpha,q}\otimes I) \text{~on $\H$}, \qquad n\geq 1, 
		\end{equation}
		where
		\begin{align*}
			R^{(n)}_{\s,\q} &= \id+\sum_{k=1}^{n-1}\q^{k}\pi_{n-1}\cdots \pi_{n-k} + \s \q^{n-1}\pi_{n-1} \pi_{n-2} \cdots \pi_{1}\pi_0\left(1+\sum_{k=1}^{n-1}\q^{k}\pi_{1}\cdots \pi_{k}\right), \quad n\geq 2,
			\\ R^{(1)}_{\s,\q} &=\id+\s\pi_0.
		\end{align*}
  and the adjoint ${R^{(n)}_{\alpha,q}}^*$ is taken with respect to $\langle\cdot,\cdot \rangle_{0,0}$.
In 	
\cite[Lemma 2.8]{BEH15}, we obtain the following estimation of norm 
\begin{align} \label{estymationR}
&\|R_{\alpha,q}^{(n)}\|_{0,0} \leq (1+|\alpha| |q|^{n-1})[n]_q,\qquad n \geq 1. 
\end{align}
\end{proposition}

		For $x\otimes y \in \HH_\R$ we define $\B^\ast(x\otimes y):=
		\r^\ast(x\otimes y)$ and we can calculate its adjoint operator with
		respect to the inner product $\langle\cdot,\cdot \rangle_{\s,\q}$ (see \cite[Proposition 3.6]{BE23})  in terms of $R_{\s,\q}^{(n)}$, namely \begin{equation}
			\B(x\otimes y)=\r(x\otimes y) R^{(n)}_{\s,\q} \text{~on $\H$}.
		\end{equation} 
  The operators $\B^\ast(x\otimes y)$ and $\B(x\otimes y)$ are called \emph{double creation and  double annihilation operator of type B, respectively}.

		Using the above notation we can decompose $\B(x\otimes y)$ into the positive part $p_\q$ and the negative part  $\ell_{\q}$  as
		$$
		\B(x\otimes y)= p_\q(x\otimes y)+ \s  \ell_{\q}(x\otimes y), \qquad x\otimes y \in \HH, 
		$$ 
		where 
		\begin{align}
			p_\q(x\otimes y)\eta &= \sum_{k=1}^n \q^{n-k}\langle x, x_{\bar{k}}\rangle  \langle y, x_k \rangle\, x_{\bar n }\otimes \cdots \otimes \check{x}_{\bar k} \otimes \cdots \otimes x_{\bar 1} \otimes  x_1\otimes \cdots \otimes \check{x}_k \otimes \cdots \otimes x_n, \label{rq}
			\\
			\ell_{\q}(x\otimes y)\eta &=\q^{n-1}\sum_{k=1}^n  \q^{k-1}\langle x, x_{{k}}\rangle  \langle y, x_{\bar{k}}\rangle\, x_{\bar n }\otimes \cdots \otimes \check{x}_{\bar k} \otimes \cdots \otimes x_{\bar 1} \otimes  x_1\otimes \cdots \otimes \check{x}_k \otimes \cdots \otimes x_n.\label{lq}
		\end{align}
		where $\eta=x_{\bar n }\otimes \cdots \otimes x_n.$

\begin{remark}
    We can observe that for $x\otimes y, y \otimes x \in \HH_\R$ and  $\eta=x_{\bar n }\otimes \cdots \otimes x_n \in \H$
    \begin{align}
    p_\q(x\otimes y)\eta = \q^{2n-2}\ell_{\q^{-1}}(y\otimes x)\eta.
    \end{align}
\end{remark}
The norm of the creation operators is calculated in  \cite[Theorem 3.9]{BE23}. 
	\begin{theorem} Suppose that $x\otimes y \in \HH_{\R}, x \otimes y \neq0$. 
		\begin{enumerate}
			\item\label{item1} If  $(\s,\q)\in A $, where $A=[0 ,1]\times (-1,0]$, then 
			\begin{equation}\label{eq01}
				\|\B^\ast(x\otimes y)\|_{\s,\q}= \sqrt{\|x\|^2\|y\|^2  + \s \langle x,y\rangle^2}. 
			\end{equation}
			\item\label{item2} If  $(\s,\q)\in B $, where $B= [-1 ,0)\times (-1,0]$, then 
			\begin{equation}\label{eq02}
				\frac{\|x\|\|y\|}{\sqrt{1-\q}}\leq \|\B^\ast(x\otimes y)\|_{\s,\q}\leq \|x\|\|y\|. 
			\end{equation}
			\item\label{item3} If $(\s,\q)\in  C $, where $C= \{(\s,\q)\mid  |\s| \leq \q <1\}$, then 
			$$
			\|\B^\ast(x\otimes y)\|_{\s,\q}= \frac{\|x\|\|y\|}{\sqrt{1-\q}}. 
			$$
			\item\label{item5} Otherwise, if $(\s,\q)\in [-1 ,1]\times (-1,1)\setminus (A\cup B\cup C) $
			$$
			\frac{\|x\|\|y\|}{\sqrt{1-\q}}\leq \|\B^\ast(x\otimes y)\|_{\s,\q}\leq \sqrt{\frac{1+|\s|}{1-\q}}\|x\|\|y\|. 
			$$
		\end{enumerate}
	\end{theorem}
\subsection{Gauge operators} In this subsection we define a second differential quantization operator.
Let $T$ and $\bar T$ be the operators in the Hilbert space $H$ with a dense domain ${D}$. 
 This in particular means
that there is no relation between operators $T$ and $\bar T$. 
We also assume  that $\bar T({ {D}} )\subset { {D} }$, $T({D}) \subset {D}$, $\mc{D}:= D \otimes D$ and $\mc{D}_n:= D^{\otimes n} \otimes D^{\otimes n}$. 
The following gauge operator  is motivated by the papers  \cite{Ans01,Ejsmont1,Ejsmont23}.
First we introduce an operator which acts on  double Fock space, as
\begin{align*}
& p_0(\T) \Omega \otimes  \Omega = 0, \\
& p_0(\T) x_{\bar n} \otimes \ldots \otimes x_{\bar 1} \otimes x_1 \otimes  \ldots \otimes x_n = \bar T(x_{\bar n}) \otimes \ldots \otimes x_{\bar 1} \otimes x_1 \otimes  \ldots \otimes T(x_n),
\end{align*}
where $\T$ is an operator on $\HH$ with dense domain $\mc{ D}. $ 
The adjoint of this operator satisfies $\langle p_0(\T)f | \zeta \rangle_{0,0}=\langle f | p_0((\T)^*)\zeta \rangle_{0,0}$, and allows us to define a gauge operator (preservation or differential second quantization).

\begin{definition}
    The \emph{gauge operator} $p(\T)$ is an operator on $\F$ with dense domain $\FD$ defined by 
\begin{align*}
& p_{\alpha,q}(\T) \Omega \otimes \Omega = 0, \\
& p_{\alpha,q} (\T) = p_0(\T) R^{(n)}_{\alpha,q}.
\end{align*}
\end{definition}
 If it is evident from the context, then let us omit in the notation parameters $\alpha$ and $q$ i.e. $p_{\alpha,q}(\T)=p(\T)$.
\noindent We also introduce the notation 
\begin{align}
& \r_q^{\T} :=p_0(\T)\Bigg(1+\sum_{k=1}^{n-1}q^{k}\pi_{n-1}\cdots \pi_{n-k}\Bigg) \label{rqT}, \\
& \ell_{q}^{N,\T}:= p_0(\T)q^{n-1}\Bigg(\pi_{n-1} \pi_{n-2} \cdots \pi_{1}\pi_0\left[1+\sum_{k=1}^{n-1}q^{k}\pi_{1}\cdots \pi_{k}\right]\Bigg), \label{lqT}
\end{align}
i.e. $p(\T)=\r_q^{\T} +\alpha \ell_{q}^{N,\T}$.
 We can rewrite the action of $\r_q^{\T}$ and $\ell_q^{N,\T}$ on $\H$ as
\begin{align*} 
\r_q^{\T} (\eta) =\sum_{k=1}^n q^{n-k} \bar T({x}_{\bar k})\otimes x_{\bar n} \otimes \ldots \otimes \check{x}_{\bar k} \otimes \ldots \otimes x_{\bar 1} \otimes x_1\otimes \cdots \otimes \check{x}_k \otimes \cdots \otimes x_n\otimes T(x_k),
\end{align*}
\begin{align*}
\ell_q^{N,\T}(\eta)&= q^{n-1}\sum_{k=1}^n q^{k-1} \bar T(x_k)\otimes x_{\bar n} \otimes \ldots \otimes \check{x}_{\bar k} \otimes \ldots \otimes x_{\bar 1} \otimes x_1\otimes \cdots \otimes \check{x}_k \otimes \cdots \otimes x_n\otimes T(x_{\bar k}), 
\end{align*}
where $\eta=x_{\bar n }\otimes \cdots \otimes x_n.$
In the last equations we used the relation  
\begin{figure}[h]
\begin{center}
 \begin{tikzpicture}[thick,font=\small,scale=.9]
     \path 
           (-8,0) node[] (bc) {$-n$}
           (-8,0.4) node[] (pom1) {$\blacktriangledown$}
           (-6.3,0) node[] (d) {{$...$}}
           (-4.5,0) node[] (g) {$-(n-k)$}
           (-2.75,0) node[] (gg) {{$...$}}
           (-2,0) node[] (gg) {$- 1$}
           (0,0) node[] (bcb) {1} 
           (5.9,0.4) node[] (pom1) {$\blacktriangledown$}
           (0.72,0) node[] (d) {{$...$}}
           (2.4,0) node[] (gd) {$n-k$}
      (4.2,0) node[] (gg) {{$...$}}
           (-1,2.5) node[] (wed) {$ \pi_{n-1}\cdots \pi_{n-k}=\bigl(\begin{smallmatrix}
  -n & -(n-1) & \cdots & -(n-k) & \cdots & -1 & 1 & \cdots & n-k & \cdots & n-1 & n \\
  -(n-k) & -n & \cdots & -(n-k+1) & \cdots & -1 & 1 & \cdots & n-k+1 & \cdots & n & n-k
 \end{smallmatrix}\bigr)$}
           (5.9,0) node[] (ggg) {$n$};
       \draw[thick] (g) -- +(0,1) -| (bc);
       \draw[thick] (ggg) -- +(0,1) -| (gd);
   \end{tikzpicture}
   \end{center}  
\begin{center}
 \begin{tikzpicture}[thick,font=\small,scale=.9]
     \path 
           (-8,0) node[] (bc) {$-n$}
           (-8,0.4) node[] (pom1) {$\blacktriangledown$}
           (-6.3,0) node[] (d) {{$...$}}
           (-4.5,0) node[] (g) {$-(n-k)$}
           (-2.75,0) node[] (gg) {{$...$}}
           (-2,0) node[] (gg) {$- 1$}
           (0,0) node[] (bcb) {1} 
           (5.9,0.4) node[] (pom1) {$\blacktriangledown$}
           (0.72,0) node[] (d) {{$...$}}
           (2.4,0) node[] (gd) {$n-k$}
      (4.2,0) node[] (gg) {{$...$}}
           (-1,2.5) node[] (wed) {$ 
\pi_{n-1} \pi_{n-2} \cdots \pi_{1}\pi_0\pi_{1}\cdots \pi_{k}=\bigl(\begin{smallmatrix}
  -n & -(n-1) & \cdots & -(k+1) & \cdots & -1 & 1 & \cdots & k+1 & \cdots & n-1 & n \\
  k+1 & -n & \cdots & -(k+2) & \cdots & -1 & 1 & \cdots  & k+2 & \cdots & n & -(k+1)
 \end{smallmatrix}\bigr)$}
           (5.9,0) node[] (ggg) {$n$};
       \draw[thick] (bc) -- +(0,1) -| (gd);
       \draw[thick] (g) -- +(0,1.5) -| (ggg);
   \end{tikzpicture}
   \end{center}  
\begin{center}
 \begin{tikzpicture}[thick,font=\small,scale=.9]
     \path 
           (-8,0) node[] (bc) {$-n$}
           (-8,0.4) node[] (pom1) {$\blacktriangledown$}
           (-6.3,0) node[] (d) {{$...$}}
           (-4.5,0) node[] (g) {$-(n-k)$}
           (-2.75,0) node[] (gg) {{$...$}}
           (-2,0) node[] (ggk) {$- 1$}
           (0,0) node[] (bcb) {1} 
           (5.9,0.4) node[] (pom1) {$\blacktriangledown$}
           (0.72,0) node[] (d) {{$...$}}
           (2.4,0) node[] (gd) {$n-k$}
      (4.2,0) node[] (gg) {{$...$}}
           (-1,2.5) node[] (wed) {$ 
\pi_{n-1} \pi_{n-2} \cdots \pi_{1}\pi_0=\bigl(\begin{smallmatrix}
   -n & -(n-1) & \cdots & -2 & -1 & 1 & 2 &\cdots & n-1 & n \\
 1 & -n & \cdots & -3 & -2 &  2 & 3 &\cdots & n & -1
 \end{smallmatrix}\bigr)$}
           (5.9,0) node[] (ggg) {$n$};
       \draw[thick] (bc) -- +(0,1) -| (bcb);
       \draw[thick] (ggk) -- +(0,1.5) -| (ggg);
   \end{tikzpicture}
   \end{center}
\end{figure}

\newpage
\begin{remark}
    Let $x_{\bar n}=x_n,\dots, x_{\bar1}=x_1$, then $\eta=x_n\otimes x_{n-1}\otimes x_1\otimes x_1 \otimes x_{n-1}\otimes x_n$ and
    \begin{align*}
    \r_q^{\T} (\eta)=q^{2n-2} \ell_{q^{-1}}^{N,\T}(\eta).
    \end{align*}
    
\end{remark}
\begin{Prop} \label{Prop:samosprzezone}If $\T$ is essentially self-adjoint on a dense domain $\mc{D}$ and $(\T)(\mc{D}) \subset \mc{D}$, then $p(\T)$ is essentially self-adjoint on the dense domain $\FD$.
\end{Prop}

\begin{proof}We first observe that $p_0((\T)^\ast) ( I \otimes P^{(n-1)}_{\alpha,q}\otimes I)= (I \otimes P^{(n-1)}_{\alpha,q}\otimes I)p_0((\T)^\ast)$, indeed for $x_{\bar n} \otimes \ldots \otimes x_n \in \mc{D}_n$, we have
\begin{align*}&p_0((\T)^\ast) (I \otimes P^{(n-1)}_{\alpha,q}\otimes I)(x_{\bar n} \otimes \ldots \otimes x_n)\\&=p_0(\bar T^\ast \otimes T^\ast) (I \otimes P^{(n-1)}_{\alpha,q}\otimes I)(x_{\bar n} \otimes \ldots \otimes x_n)\\&= \bar T^*(x_{\bar n}) \otimes  P^{(n-1)}_{\alpha,q}(x_{\overline{n-1}} \otimes \ldots \otimes x_{n-1})\otimes T^*(x_n)\\&=   (I\otimes P^{(n-1)}_{\alpha,q}\otimes I)(\bar T^\ast (x_{\bar n}) \otimes \ldots \otimes T^*(x_n))\\&=(I \otimes P^{(n-1)}_{\alpha,q}\otimes I)p_0(\bar T^\ast\otimes T^\ast) (x_{\bar n} \otimes \ldots \otimes x_n)\\&=(I \otimes P^{(n-1)}_{\alpha,q}\otimes I)p_0((\bar T\otimes T)^\ast) (x_{\bar n} \otimes \ldots \otimes x_n)
.
\end{align*}
Now we show that $p(\T)$ is symmetric on $\FD$. Let us fix $n$, and $f,g \in \mc{D}_n$, then
\begin{align*}
\ip{p(\T) f}{g}_{\alpha,q} &=\langle p(\T) f, P^{(n)}_{\alpha,q} g \rangle_{0,0} \\&=\langle p_0(\T) R^{(n)}_{\alpha,q} f,  (I\otimes 
 P^{(n-1)}_{\alpha,q}\otimes I) R^{(n)}_{\alpha,q}g \rangle_{0,0} \\
&= \langle R^{(n)}_{\alpha,q} f, p_0((\T)^\ast) (I \otimes P^{(n-1)}_{\alpha,q}\otimes I) R^{(n)}_{\alpha,q}g \rangle_{0,0}
 \\&= \langle R^{(n)}_{\alpha,q} f,  (I \otimes P^{(n-1)}_{\alpha,q}\otimes I)p_0((\T)^*)) R^{(n)}_{\alpha,q}g \rangle_{0,0}
\intertext{by Proposition \ref{prop1}, we have}
&=\langle  f,  {R^{(n)}_{\alpha,q}}^*(\I \otimes P^{(n-1)}_{\alpha,q}\otimes I)p_0((\T)^*) R^{(n)}_{\alpha,q}g \rangle_{0,0}
= \ip{f}{p((\T)^\ast)g}_{ \alpha,q}.
\end{align*}

\noindent
Now we show that $\T$ is essentially self-adjoint. Let $E$ be the spectral measure of the closure of $\T$ and $C \in \mf{R}_+$. Let $\set{x_i}_{i=\bar n}^n \subset (E_{[-C. C]} H) \cap \mc{D}$, $\vec{x} = x_
{\bar n}\otimes \ldots \otimes x_n$, then $\norm{\T (x_{\bar i}\otimes x_i)} \leq C \norm{x_{\bar i}\otimes x_i}$ and
\begin{align*}
&\ip{p(\T) \vec{x}}{ p(\T) \vec{x}}_{0,0}=\ip{p_0(\T){R^{(n)}_{\alpha,q}} \vec{x}}{ p_0(\T){R^{(n)}_{\alpha,q}} \vec{x}}_{0,0}\leq C^2 \norm{{R^{(n)}_{\alpha,q}}\vec{x}}_{0,0}^2 
\intertext{by  \cref{estymationR}, we have}
&\leq \big(C(1+|\alpha| |q|^{n-1})[n]_q \norm{\vec{x}}_{0,0}\big)^2 
\leq \big(2nC \norm{\vec{x}}_{0,0}\big)^2.
\end{align*}
Thus we get the following estimation for the norm of $p(\T)^k $
\begin{align*}
&\norm{p(\T)^k \vec{x}}_{ \alpha,q}^2 = \ip{p(\T)^k \vec{x}}{P_{ \alpha,q}^{(n)} p(\T)^k \vec{x}}_{0,0} 
\leq \norm{P_{ \alpha,q}^{(n)}}_{0,0}^2\ip{p(\T)^k \vec{x}}{ p(\T)^k \vec{x}}_{0,0}\\&
\leq \norm{P_{ \alpha,q}^{(n)}}_{0,0}^2 (2^kn^k C^k \norm{\vec{x}}_{0,0})^2.
\end{align*}
Now, we use the estimations from the proof of \cite[Theorem 2.9, equation (2.32)]{BEH15} i.e. 
\begin{equation}
\begin{split}
P_{\alpha,q}^{(n)} \leq (1+|\alpha\|q|^{n-1})[n]_q (\I \otimes P_{\alpha,q}^{(n-1)}\otimes I) ,
\end{split}
\end{equation}
with respect to the $(0,0)$-inner product, so $\norm{P_{ \alpha,q}^{(n)}}_{0,0} \leq \prod_{i=1}^n (1+|\alpha||q|^{i-1})[i]_{q} \leq 2^n n!$. It can be shown that $\norm{p(\T)^k \vec{x}}_{\alpha,q} \leq {2^n n!}2^k n^k C^k \norm{\vec{x}}_{0,0}$, so the  series $\sum_{k=0}^\infty\frac{p(\T)^k \vec{x}}{k!}s^k$ has a positive radius of absolute convergence, because 
\[
\limsup_{k \rightarrow \infty}\sqrt[k]{ \frac{\norm{p(\T)^k \vec{x}}_{ \alpha,q}}{k!}}\leq \limsup_{k \rightarrow \infty}\sqrt[k]{ \frac{{2^n n!} 2^kn^k C^k \norm{\vec{x}}_{0,0}}{k!}} = 0.
\]
Therefore $\vec{x}$ is an analytic vector for $p(\T)$. The linear span of such vectors is invariant under $p(\T)$ and is a dense subset of $\mc{D}^{\otimes n}$. Therefore by Nelson's analytic vector theorem \cite{Nel59} (see also \cite{ReeSim1}), $p(\T)$ is essentially self-adjoint on $\mc{D}^{\otimes n}$.
\end{proof}

\begin{Prop} If $\T$ is a bounded operator on $\HH$, then $p(\T)$ is a bounded operator on the double Fock space of type B.
\label{Ansh+}
\end{Prop}
\begin{proof} We begin by showing that $p(\T)$ is bounded on the $\mathcal{F}_{0,0}(H)$. Next we show that $p(\T)$ is bounded. Using  \cref{Prop:samosprzezone} it can be shown that $P_{\alpha,q}p((\T)^{*})=p(\T)^{*}P_{\alpha,q}$, where $p(\T)^{*}$ is taken with respect to the $(0,0)$-inner product. Indeed, for $f,g \in H^{\otimes n}$, we have  \begin{align*}
&\langle f, p((\T)^*) g \rangle_{\alpha,q}= \ip{ f}{P^{(n)}_{\alpha,q} p((\T)^*)g}_{0,0}
\intertext{ by  \cref{Prop:samosprzezone}, we get} 
&=\ip{p(\T) f}{g}_{\alpha,q} =\langle p(\T) f, P^{(n)}_{\alpha,q} g \rangle_{0,0}= \langle f, p((\T)^*)P^{(n)}_{\alpha,q} g \rangle_{0,0}.
\end{align*}
This gives us $P_{\alpha,q}p((\T)^{*})p(\T)=p((\T)^{*})P_{\alpha,q}p(\T)\geqslant 0 $ and 
$$
P_{\alpha,q}p((\T)^{*})p(\T) [p((\T)^{*})p(\T)]^*P_{\alpha,q}$$
$$\leqslant \|p((\T)^{*})p(\T) [p((\T)^{*})p(\T)]^*\|_{0,0}P_{\alpha,q}^2.
$$
By taking the square root of the operators from above inequality, we get
\begin{align*} \nonumber
P_{\alpha,q}p((\T)^{*})p(\T)\leqslant \sqrt{\|p((\T)^{*})p(\T) [p((\T)^{*})p(\T)]^*\|_{0,0}}P_{\alpha,q} \\
\leqslant \|p((\T)^{*})\|_{0,0}\|p(\T)\|_{0,0}P_{\alpha,q}. \label{eq:nierownosc1}
\end{align*}
If we take $f\in \H$, then we get
\begin{align*} \nonumber
&\langle p(\T)f | p(\T)f \rangle_{\alpha,q} = \langle f | p((\T)^*)p(\T)f \rangle_{\alpha,q} = \langle f | P_{\alpha,q}p((\T)^*)p(\T)f \rangle_{0,0}. 
&
\end{align*}
It is clear by the definition of $p_{0}$ that $\|p_{0}\|_{0,0}\leqslant \|\T\|$, and thus 
\begin{align*} &\|p(\T)f\|_{0,0}=\|p_0(\T)R^{(n)}_{\alpha,q} f\|_{0,0}\leqslant \|p_0\|_{0,0}\| R^{(n)}_{\alpha,q}f\|_{0,0}.
\intertext{Now we use the estimation from \cref{estymationR} and we get}
& \leqslant \|\T\|(1+|\alpha| |q|^{n-1})[n]_q\|f\|_{0,0} \leqslant \max\{1+|\alpha| ,(1+|\alpha| )/(1-q)\}\|\T\|\|f\|_{0,0}. 
\end{align*}
Finally, since $\|(\T)^*\|=\|\T\|$, we conclude that
$$
\| p(\T) \|_{\alpha,q} \leqslant \sqrt{\|p((\T)^*)\|_{0,0}\|p(\T)\|_{0,0}} \leqslant (1+|\alpha| )\max\{1,1/(1-q)\}\|\T\|.
$$
\end{proof}


	\section{The Double Poisson operator of type B }
Now we provide an explicit formula for the combinatorial mixed moments, involving the number of  crossings and negative arcs of a partition. First, we need to define the operators, the set of partitions and statistics of type B.


 Now  we  define differential second quantization operator on $\FQ$. 
In order to do this, we introduce some special operators.  We define $T_{\bar i}:=\bar T_i$ and
$\lambda_1 \otimes \lambda_{\bar 1}$ to be $\lambda_1 \lambda_{\bar 1}\I \otimes \I.$
 
	\subsection{The orthogonal polynomials}

 \begin{definition}
		The operator 
		\begin{equation}
			\G^{\lambda_{\bar i},\lambda_i}(x_{\bar i}\otimes x_i)= \B( x_{\bar i
   }\otimes x_i) +\B^\ast(x_{\bar i}\otimes x_i)+p( T_{\bar i} \otimes T_i)+\lambda_{\bar i}\otimes\lambda_i ,\qquad x_{\bar i} \otimes x_i\in \HH_\R, \quad \lambda_{\bar i},\lambda_i\in\R,
		\end{equation}
		on $\F$ is called the \emph{double Poisson operator of type B}. 
		Denote by $\state$ the vacuum vector state $\state(\cdot)=\langle\Omega \otimes  \Omega, \cdot\text{ } \Omega\otimes  \Omega\rangle$. To simplify the notation we define $\G(x_{\bar i} \otimes x_i):=\G^{0,0}(x_{\bar i} \otimes x_i)$.
		
	\end{definition}

	Using a similar argument as in \cite[Theorem 3.3]{BEH15}, we can prove the following. For the reader's convenience, we provide this proof. 
	
	\begin{theorem} Suppose $\s,\q\in(-1,1)$ and $x\otimes y \in \HH_\R, \|x\|=\|y\|=1$ and $T=\bar T=\I$. Let $\nu_{\alpha,q,x,y}$ be the probability distribution of $\GZ(x\otimes y)$ with respect to the vacuum state. Then $\nu_{\alpha,q,x,y}$ is equal to $\mu_{\s\langle x,y\rangle^2,\q}$. 
	\end{theorem}
	\begin{proof} 
		We observe that for $n \geq 1$, we have 
		\begin{align*}
			\GZ(x\otimes y )x^{\otimes n}\otimes y^{\otimes n} 
			&= \B^\ast(x\otimes y )x^{\otimes n}\otimes y^{\otimes n} +\B(x\otimes y) x^{\otimes n}\otimes y^{\otimes n}+p(\bar T \otimes T)(x^{\otimes n}\otimes y^{\otimes n})  \\
			&= x^{\otimes  (n+1)}\otimes y^{\otimes  (n+1)}+[n]_\q (1 + \s\langle x,y\rangle^2 \q^{n-1})x^{\otimes (n-1)}\otimes y^{\otimes (n-1)}
   \\& +[n]_\q (1 + \s\langle x,y\rangle^2 \q^{n-1})x^{\otimes n}\otimes y^{\otimes n},
		\end{align*}
	Note that for $n=1$ $$Q_{1}^{(\s\langle x,y\rangle^2,\q)}(\GZ(x\otimes y ))\Omega\otimes \Omega =\GZ(x\otimes y )\Omega \otimes \Omega =x\otimes y$$ and by induction 
		\begin{align*}
			Q_{n+1}^{(\s\langle x,y\rangle^2,\q)}(\GZ(x\otimes y ))\Omega \otimes \Omega & =\GZ(x\otimes y ) Q_n^{(\s\langle x,y\rangle^2,\q)}( \GZ(x\otimes y ))\Omega\otimes \Omega \\& - [n]_q (1 + \s\langle x, y\rangle^2 q^{n-1})Q_{n-1}^{(\s\langle x,y\rangle^2,\q)}( \GZ(x\otimes y ))\Omega\otimes \Omega \\& - [n]_q (1 + \s\langle x, y\rangle^2 q^{n-1})Q_{n}^{(\s\langle x,y\rangle^2,\q)}( \GZ(x\otimes y ))\Omega\otimes \Omega \\&=\GZ(x\otimes y ) x^{\otimes n}\otimes y^{\otimes n} - [n]_q(1+ \s\langle x, y\rangle^2 q^{n-1})x^{\otimes (n-1)}\otimes y^{\otimes (n-1)}
			\\ &- [n]_q(1+ \s\langle x, y\rangle^2 q^{n-1})x^{\otimes n}\otimes y^{\otimes n} =x^{\otimes n+1}\otimes y^{\otimes n+1}.
		\end{align*}

		Therefore, the map $\Phi\colon (\text{span}\{x^{\otimes n}\otimes y^{\otimes n}\mid n \geq 0\}, \|\cdot\|_{\s,\q}) \to L^2(\R,\mu_{\s\langle x,y\rangle^2 ,\q})$ defined by $\Phi(x^{\otimes n}\otimes y^{\otimes n})= Q_n^{(\s\langle x,y\rangle^2,\q)}(t)$ is an isometry. 
		Since $\Phi$ is an isometry, we get 
		$\langle \Omega \otimes \Omega, \GZ(x\otimes y)^n\Omega\rangle_{\s,\q} = m_n(\mu_{\s\langle x,y\rangle^2,\q})$, where $m_n(\mu)$ is the $n$-th moment of measure $\mu$. 
		Since $\mu_{\s\langle x,y\rangle^2,\q}$ is a compactly supported measure, the Hamburger moment problem has a unique solution and hence $\mu_{\s\langle x,y\rangle^2,\q}=\m$.  
	\end{proof}

	\subsection{Partitions of type B}
\label{subsec:PartitionsB}
We are interested  in 
	$\ast$-algebra generated by $\B(x\otimes y) ,\B^\ast(x\otimes y)$ and $p(\T)$. Particularly interesting are their  mixed moments.
	In order to work effectively on this object we need to define the corresponding partition and statistics.
 
	Let $S$ be an ordered set. Then $\pi = \{ V_1,\dots, V_p\}$ is a partition of $S$, if the $V_i \neq \emptyset$ are
	ordered and disjoint sets $V_i=(v_1,\dots,v_k)$, where $v_1<\dots<v_k$, whose union is $S$.
	For  $V \in \pi$ ,  we say that $V$ is a \emph{block of $\pi$}.
	Any partition $\pi$ defines an equivalence relation on $S$,
	denoted by $\sim_\pi$, such that the equivalence classes are the 
	blocks of $\pi$. 
	Therefore, $i\sim_\pi j$ if $i$ and $j$ belong to the same block of $\pi$. 
	A block of $\pi$ is called a \emph{singleton} if it consists of one element. 
	Similarly, a block of $\pi$ is called a \emph{pair} if it consists of two elements.
	Let $\Sing(\pi)$ and $\Pair(\pi)$  denote the set of all singletons and blocks of $\pi$, respectively.   
	The set of  partitions of $S$ is denoted by $\Part(S)$, in the case where $S =
	[n] := \{1, \dots , n\}$ we write $\Part(n)$ := $\Part([n])$. 
	We denote by $\P_{1,2}(S)$ the subset of partitions $\pi \in \Part(S)$ whose every block is either a 
	pair or a singleton. The set of partitions without singletons is denoted by $\P_{\geq 2}(n)$ and the set of  partitions whose every block is a pair by $\P_2(n).$ 
 
	Given a partition $\pi$ of the set $[n]$, we write $\Semi(\pi)$ for the set of pairs of integers $(i, j)$ which occur in the same block of $\pi$ such that $j$ is the smallest element of the block greater than $i$. The same notation $\Semi(V)$ is applied to a block $V\in \pi$ - see Def. \ref{def:arcs}.
 
	From now on, we will work on a  set  $[\pm n]:=\{\bar n, \dots , \bar 1, 1,\dots  , n\}.$
	For a block $V=(a_1, \dots, a_k)$ (or a singleton $V=(a)$) we denote its reflection by $\overline V=(\bar  a_k , \dots, \bar a_1)$ ($\overline V=( \bar a)$), where $\bar a =-a$.   
	Similarly, we define 
	$$\bar \pi: =\{\bar V\mid V\in \pi,\text{ } \pi \in \P([\pm n])\}.$$ 
	 When we draw the points of a block then we think that consecutive elements in every block (bigger than one) are connected by arcs above the $x$ axis (see Fig. \ref{fig:arcs}).

  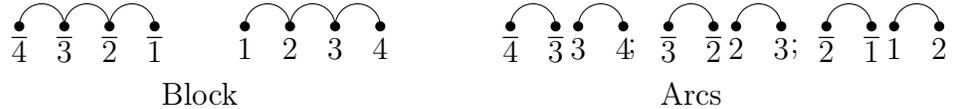
\begin{figure}[h]
		\begin{center}
  \MatchingMeandersabc{4}{-4/-3, -3/-2, -2/-1, 1/2, 2/3, 3/4}{}
  \hspace{3cm}
		  \begin{minipage}{.2\textwidth}
            \begin{tikzpicture}[scale=0.6]
			\draw(-4,0) arc[radius=0.5, start angle=180, end angle=0];
            \draw[circle,fill](-4,0)circle[radius=1mm];
            \draw[circle,fill](-3,0)circle[radius=1mm];
		\node at (-4,-0.5) { $\overline 4$}; 
		\node at (-3,-0.5) { $ \overline 3$};

  \draw(-2.5,0) arc[radius=0.5, start angle=180, end angle=0];
            \draw[circle,fill](-2.5,0)circle[radius=1mm];
            \draw[circle,fill](-1.5,0)circle[radius=1mm];
		\node at (-2.5,-0.5) { $3$}; 
		\node at (-1.5,-0.5) { $4$};
  \node at (-1.3,-0.5) { $;$};
  
  \draw(-0.5,0) arc[radius=0.5, start angle=180, end angle=0];
            \draw[circle,fill](-0.5,0)circle[radius=1mm];
            \draw[circle,fill](0.5,0)circle[radius=1mm];
		\node at (-0.5,-0.5) { $\overline 3$}; 
		\node at (0.5,-0.5) { $\overline 2$};
  
  \draw(1,0) arc[radius=0.5, start angle=180, end angle=0];
            \draw[circle,fill](1,0)circle[radius=1mm];
            \draw[circle,fill](2,0)circle[radius=1mm];
		\node at (1,-0.5) { $2$}; 
		\node at (2,-0.5) { $3$};
   \node at (2.3,-0.5) { $;$};

    \draw(3,0) arc[radius=0.5, start angle=180, end angle=0];
            \draw[circle,fill](3,0)circle[radius=1mm];
            \draw[circle,fill](4,0)circle[radius=1mm];
		\node at (3,-0.5) { $\overline 2$}; 
		\node at (4,-0.5) { $\overline 1$};
  
  \draw(4.5,0) arc[radius=0.5, start angle=180, end angle=0];
            \draw[circle,fill](4.5,0)circle[radius=1mm];
            \draw[circle,fill](5.5,0)circle[radius=1mm];
		\node at (4.5,-0.5) { $1$}; 
		\node at (5.5,-0.5) { $2$};
  \node at (0,-1.5) {Arcs};
		\end{tikzpicture}
 \end{minipage}
  \end{center}  
		\caption{The example of a block and corresponding arcs.}
		\label{fig:arcs}
	\end{figure}

	\begin{definition} \label{def:partycji} 
		We denote by $\P^{B}(n)$ the subset of partitions $\pi \in \P([\pm n])$
		such that they are symmetric $\overline{\pi} =
		\pi$ (which is invariant under the bar operation), but every arc $V \in \Semi(\pi)$ is different from its reflection
		$\overline{V}$, i.e.,~$V \neq \overline{V}$. 
		We call  $\P^{B}(n)$ the set of partitions of type B. 
	\end{definition}

	From Definition \ref{def:partycji} it follows that for every block in $B\in \pi$, $\pi \in \P^{B}(n)$ there exists a unique {reflection block} $\bar B\in\pi$.
	This  leads to one more definition. We call  $\BL= ((\bar a_k, \dots, \bar a_1),( a_1, \dots, a_k))$ a \emph{B-block} if $\bar a_k  <a_k$, $|{a_1}|  <|a_k|$ and $(\bar a_k, \dots, \bar a_1 ), (a_1, \dots, a_k)\in \pi$, $a_k\leq a_n$. The set of B-blocks  is denoted by $\Pair_B(\pi)$. All B-blocks $((\bar a_k, \dots, \bar a_1), ( a_1, \dots, a_k))$ consist of two blocks - the \emph{positive} one $( a_1, \dots, a_k)$ with $ a_k>0$ and the \emph{negative} one $(\bar a_k, \dots, \bar a_1)$ with $ \bar a_k<0$.

\begin{definition}
    Let $\pi\in\PB(n)$, then we denote by $ \BL\in \Pair_B(\pi)$ each block of type B, where $\BL=((\bar i_m,\dots,\bar i_1),(i_1,\dots,i_m))$  and 
\begin{align*}&\Semi(\BL)=\big\{((\bar i_{j+1},\bar i_j),(i_j,i_{j+1}))\}\mid j\in[m]\big\}, \quad (m\geq 2),
\intertext{and} 
&\Semi_B(\pi)=\cup_{\BL\in\Pair_B(\pi)}\Semi(\BL),
\end{align*}
 i.e. it is a set of arcs.
 \label{def:arcs}
\end{definition}
We call that the pair $((\bar i_{j+1},\bar i_j),(i_j,i_{j+1}))$ is the \emph{positive} arc if both $i_j,i_{j+1}>0$ or the negative arc if $i_j<0$ and $i_{j+1}>0$ (see Fig. \ref{fig:coneced}). Sometimes we call B-arcs instead of arcs.
 
	Especially, we call    $\BL= ((\bar a_2, \bar a_1),( a_1,a_2))$ a   \emph{B-pair} if  $\bar a_2  <a_2$, $|{a_1}|  <|a_2|$ and $(\bar a_2,\bar a_1 ), (a_1,a_2)\in \pi$. For more information about B-pairs see \cite{BE23}. %
	
	Let us notice that from Definition \ref{def:partycji} it follows that 
	the element $s$ is a singleton of $\pi \in \P_{1,2}^{B}(n)$ if and only if $\bar s$ is also a singleton of $\pi,$ thus we can define the subset of  B-singletons $((\bar a),(a))$  as  
	$$\Sing_B(\pi):=\{((\bar a), (a)) \mid \bar a,a \in \Sing(\pi), \text{ }\bar a < a\}, \quad  \pi \in \P_{1,2}^{B}(n).$$ 
Then we can define the set of type-B partitions without singletons as
\begin{align*}
\P_{\geq 2}^B(n)&:=\{\pi \in \P^{B}(n)\mid \Sing_B(\pi)=\emptyset\} 
\end{align*}
and $\PA(n)$ as a subset of $\P^{B}(n)$ with only positive arcs.


	
	\begin{remark}
		(1). We note that a   B-block is not a block of $\pi$, but a pair of reflection blocks.
		
		\noindent  (2). We note that $\#\PA(n)={\displaystyle \sum _{i=0}^{k}{\begin{Bmatrix}k\\i\end{Bmatrix}},}$ where the braces $\{ \}$ denote Stirling numbers of the second kind, which is the same as the number of the classical  partitions on $n$ points. 
	
 	
	\end{remark}

	\begin{figure}[h]
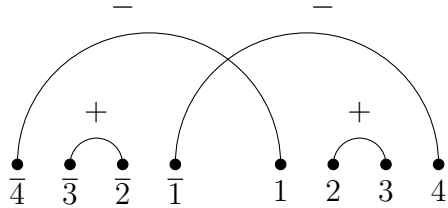

		\begin{center}
			\MatchingMeandersarcs{4}{-4/1, -3/-2,2/3,-1/4 }{} 

		\end{center}  
		\caption{The example of $\pi \in \P^{B}(4)$ with a positive arc ($+$) - $((\bar 3,\bar 2 ), (2,3))$ and a negative  arc ($-$) - $((\bar 4,1), (\bar 1,4))$.}
		\label{fig:coneced}
	\end{figure}

\begin{remark}
Note that in Fig. \ref{fig:coneced}, B-blocks are also B-pairs.
\end{remark}
 
	We introduce some partition statistics.

	\noindent For two arcs  $V$ and $ W$, we introduce the relation  $\text{rc}$ as follows: 
	\begin{align*}
		V\stackrel{\text{rc}}{\sim}W 
		&\iff
		 V=arc(i,j), \textrm{ } W=arc (i',j')  \textrm{ and }   i < i' < j < j'  \textrm{ or } i' < i < j' < j.
	\end{align*}
	For a set partition $\pi\in \P^{B}(n)$ let $\Cr(\pi)$ be the number of crossings arcs i.e.\ 
	\begin{align*}
		\Cr(\pi)=&\#\{\big((\bar V,V), (\bar W, W)\big)\in \Semi_B(\pi)\times \Semi_B(\pi)  \mid &  \text{$V\stackrel{\text{rc}}{\sim}W $}\}\\ & +\#\{\big((\bar V,V), (\bar W, W)\big)\in \Semi_B(\pi)\times \Semi_B(\pi)  \mid &  \text{$\bar V\stackrel{\text{rc}}{\sim}W $}\}
	\end{align*}
 Note that in the definition the signs (positivity and negativity) of arcs are not important. We use the same definition of \emph{restricted crossings} as given in Biane \cite{Bia97}.

 \noindent
	For arc $V=arc(i,j)$ and block $W$, we say that arc $V$ \emph{covers} block $W$ if there are $i,j$ such that $i <k <j$ for any $k\in W$ and then we write $V\stackrel{\text{cs}}{\sim}W $. For $\pi\in\P^{B}(n)$, let $\InS(\pi)$ 
	be the number of pairs of
	a B-singleton and a covering B-arc:
	\begin{align*}
\InS(\pi)&=\#\{\big((\overline{ V}, V ),(\overline{ W}, W)
\big) \in \Semi_B(\pi) \times \Sing_B(\pi)  \mid & V\stackrel{\text{cs}}{\sim}W
\}
\\ & +\#\{\big((\overline{ V}, V ),(\overline{ W}, W)\big) \in \Semi_B(\pi) \times \Sing_B(\pi)   \mid & \overline{V}\stackrel{\text{cs}}{\sim}W 
\}.
\end{align*}
We say that the B-arc  $(\overline{ V} , V)$  cover the B-singleton $(\bar W, W)$   if $V\stackrel{\text{cs}}{\sim}{ W} $ or $\bar V\stackrel{\text{cs}}{\sim}{ W} $.

Let $\NB(\pi)$ be the number of negative B-arcs  of $\Semi_B(\pi)$, where $\pi \in \P^{B}(n)$. 
The set of arcs is \emph{noncrossing} if $\Cr(\pi)=0$. The set of noncrossing arcs is denoted by $\NC^B(n),$ and by $\NC^A(n)$ we denote the subset of   $\NC^B(n)$ where all arcs are  positive.

\begin{remark} \label{jakliczyc}
(1). While reading this, the first impression seems to be that the procedure of counting the crossings is complicated. This is not true since it can be read from the figure as follows.  First,  we draw  a vertical line in the center as in  \cref{fig:examplejakliczyccrosingi}. 
Then we count only the crossings on the left (yellow points on Fig. \ref{fig:examplejakliczyccrosingi}) of this vertical line.
We count   the number of 
negative B-arcs as the number of  B-arcs crossed by a vertical line ($\scriptstyle{\blacklozenge}$ points on \cref{fig:examplejakliczyccrosingi}). 
Similarly we count the pairs of the form $$(\text{covering B-arc, B-singleton})$$  as the number of such a pairs with at least one leg   on the left of this vertical line. 
\begin{figure}[htp]
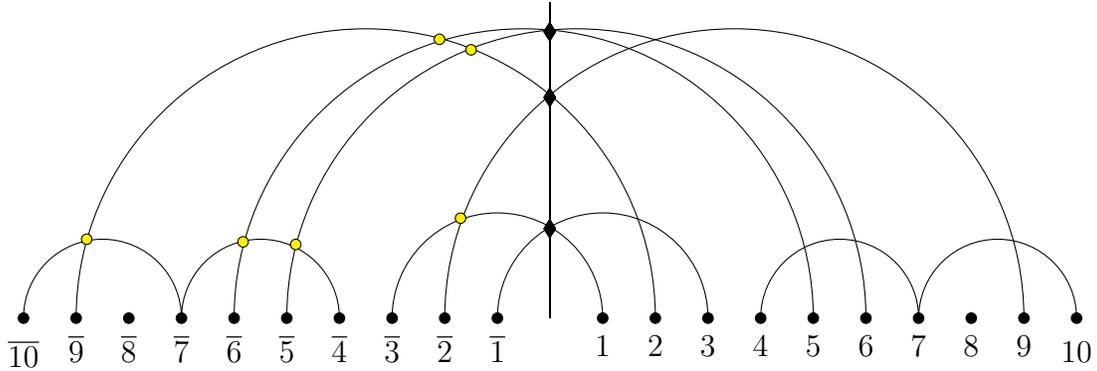

	\begin{center}
		\MatchingMeanderscol{10}{-1/3, -3/1 ,-6/5,-5/6,7/10,-10/-7,-2/9,-9/2, 4/7, -7/-4}
	\end{center}  
	\caption{The example of statistic of partition $\pi \in \P^{B}(10),$ i.e., $\Cr(\pi)=6$,  $\NB(\pi)=3$,
		$\InS(\pi)=2$.
	}
	\label{fig:examplejakliczyccrosingi}
\end{figure}
\end{remark}

\noindent   (2).  Let observe that we can introduce the following projectors
\begin{align*}
\begin{aligned}
\D: \NC^B(2 m)&\to\NC(2 m)\\
\tilde   \pi & \mapsto \pi,
\end{aligned}
\end{align*}
where $\pi$ is defined through the  
B-arcs of $\tilde \pi$; more precisely, the B-arcs  $C=((\bar b,\bar a ), (a,b)) \in \Pair_B(\tilde \pi)$ are mapped according to the action of
the following relation  
\begin{align*}
\begin{aligned}
C&\mapsto\begin{cases}
((\bar b,\bar a ), (a,b)), &  \text{if $C$ is the positive arc of }  \Pair_B(\tilde \pi), \\
((\bar b, a ), (\bar a,b)), &  \text{if $C$ is the negative arc of }  \Pair_B(\tilde \pi).
\end{cases}
\end{aligned}
\end{align*}
-- see   \cref{fig:exmapleouter} for example.   In particular we see that $\NC^B(n)=\bigsqcup_{\pi \in \NC^A(n)} \D^{-1}(\pi)$ and 
\begin{align}
\label{liczbapartycji}
    \#\D^{-1}(\pi)&=\sum_{i=0}^{\#\Out(\pi)}{\#\Out(\pi) \choose i}, \quad\pi\in \NC^A(n).  
\end{align} 

\begin{figure}[h]
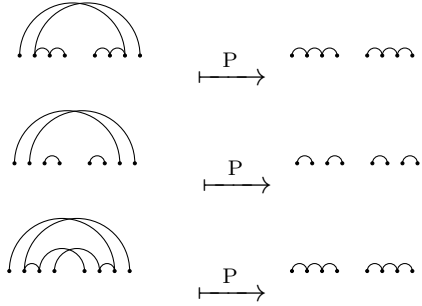

\begin{center}
\MatchingMeandersexample{4}{-4/3, -3/-2,2/3,-3/4,1/2,-2/-1 }{}\hspace{0.7em} 
$\xmapsto[\text{ }\text{ }\text{ }\text{ }\text{ }\text{ }\text{ }]{\D}$
\MatchingMeandersexample{4}{-4/-3, -3/-2,2/3,3/4,1/2,-2/-1 }{} \\   
\MatchingMeandersexample{4}{-4/3, -2/-1,1/2,-3/4 }{} \hspace{0.7em} 
$\xmapsto[\text{ }\text{ }\text{ }\text{ }\text{ }\text{ }\text{ }]{\D}$
\MatchingMeandersexample{4}{-4/-3, -2/-1,1/2,3/4 }{}
\\ 
\MatchingMeandersexample{4}{-4/3, -2/1,-1/2,-3/4,-3/-2,2/3 }{} \hspace{0.7em} 
$\xmapsto[\text{ }\text{ }\text{ }\text{ }\text{ }\text{ }\text{ }]{\D}$
\MatchingMeandersexample{4}{-4/-3, -2/-1,1/2,3/4 ,-3/-2,2/3}{}  

\end{center}
\caption{The visualization of the action $D$ on $\NC_2^B(4)$.
\label{fig:exmapleouter}
}
\end{figure} 
\begin{remark}
    
It turns out that 
the inversions and co-inversions in the symmetric group become
the joint statistics of crossings and negative arcs on the set of partitions that appear in \cref{eq:multimomentintro}. 
The symmetrization operator for $ -1\leq  \alpha, q \leq 1$, has the form 
 \begin{align*}
\sum_{\sigma \in B(n)}\s^{\text{number of negative inversions in }\sigma} \q^{\text{number of positive  inversions  in }\sigma}\sigma, 
 \end{align*}
 where \begin{itemize}
 \item $B(n)$ is the hyperoctahedral group, see Section \ref{sec:preliminaries} above;
\item  the definition of positive and negative inversion is given below.
\end{itemize}
It turns out that combinatorial interpretation of calculating the positive and negative inversions is well compatible with the procedure of counting crossings and negative arcs.
Below by using the language of the positive  and negative inversions, we rewrite a symmetrization operator; see \cite[Page 219, 220]{BozejkoSzwarc2003}. Firstly, we show  definitions  of positive and negative roots of two types: 
\begin{align*}
R_1^+&:=\{1,\dots, n  \} \text{ and }
R_1^-=-R_1^+,
\\ 
R_2^+&:=\{(i,j)\mid 1\leq i < j\leq n  \}\cup \{(i,-j)\mid 1\leq i < j\leq n  \},
\text{ and }
R_2^-
:=-R_2^+. 
\end{align*}

In this style we may define the length functions $l_1$ and $l_2$  as: 
\begin{itemize}
\item {  $l_1(\sigma)=\ninv(\sigma):=\mbox{card}\{i\mid 1\leq i \leq n\text{ and } \sigma(i )<0  \}$ } is the number of negative inversions of  $\sigma\in B(n)$;
\item  {$l_2(\sigma)=\inv(\sigma):=\mbox{card}\{(i,j)\in R_2^+\mid (\sigma(i),\sigma(j))\in R_2^- \}$ } is the number of positive inversions of  $\sigma\in B(n)$;
\end{itemize}
With the above notation we obtain 
$
P_{\s,\q}^{(n)}= \sum_{\sigma \in B(n)}\s^{\ninv(\sigma)} \q^{\inv(\sigma)} \, \sigma. 
$
This new definition of length has the following interpretation.  
Firstly we draw arrows which show the action of permutation -- see  \cref{permutacjeexmpale}. 
Then we follow the rules:
\begin{itemize}
\item if an arrow $a \rightarrow b$ crosses  $-a \rightarrow -b$, then we count them as a negative inversion. This crossing appears  in the center of picture and corresponds to  a crossing which appears when we obtain the negative B-arc (for example see the arcs crossed by a vertical line at the \cref{fig:examplejakliczyccrosingi}).  
\item if an arrow $a \rightarrow b$ crosses  $c \rightarrow d$, ($c\neq -a$) and repeatedly,  $-a \rightarrow -b$ crosses  $-c \rightarrow -d$, then we count one of them as a positive inversion. In other words,  we count positive inversions as  crossings of arrows, which lie to the  left or  to the right of the center of the picture. Similarly we count the number of crossing arcs ($\Cr(\pi)$ for $\pi \in \P^{B}(n)$). 
\end{itemize}

\begin{figure}[htp]

\begin{tikzpicture}[node distance={15mm}, thick, main/.style = {}] 
\node[main] (2)  {$-2$}; 
\node[main] (3) [ right of=2] {$-1$}; 
\node[main] (4) [ right of=3] {$1$}; 
\node[main] (5) [ right of=4] {$2$}; 
\node[main] (6) [ right of=5] { };

\node[main] (b) [below  of=2] {$-2$}; 
\node[main] (c) [ below of=3] {$-1$}; 
\node[main] (d) [ below of=4] {$1$}; 
\node[main] (e) [ below of=5] {$2$}; 
\node[main] (k) [ below of=e] {};
\node at (7.5, 0) (h) {$\sigma=
\bigl(\begin{smallmatrix}
-2 & -1 & 1  & 2  \\
2 & -1 & 1  & -2
\end{smallmatrix}\bigr)=\pi_1\pi_0\pi_1
$};
\node at (7.5, -1.5) {${\inv(\sigma)=2,\hspace{0.5em} \ninv(\sigma)=1}$ };

\draw[->] (2) -- (e); 
\draw[->] (5) -- (b); 
\draw[->] (3) -- (c); 
\draw[->] (4) -- (d); 
\end{tikzpicture}


\hspace{13em}
\begin{tikzpicture}[node distance={15mm}, thick, main/.style = {}]
\node[main] (2)  {$-2$}; 
\node[main] (3) [ right of=2] {$-1$}; 
\node[main] (4) [ right of=3] {$1$}; 
\node[main] (5) [ right of=4] {$2$}; 
\node[main] (6) [ right of=5] { };

\node[main] (b) [below  of=2] {$-2$}; 
\node[main] (c) [ below of=3] {$-1$}; 
\node[main] (d) [ below of=4] {$1$}; 
\node[main] (e) [ below of=5] {$2$}; 
\node[main] (k) [ below of=e] {};
\node at (7.5, -1.5)  (h) {${\inv(\sigma)=1,\hspace{0.5em} \ninv(\sigma)=2}$ };
\node at (7.6,-0)  (h) {$\sigma=
\bigl(\begin{smallmatrix}
-2 & -1 & 1  & 2  \\
1 & 2 & -2  & -1
\end{smallmatrix}\bigr)=\pi_0\pi_1\pi_0
$ };
\draw[->] (2) -- (d); 
\draw[->] (5) -- (c); 
\draw[->] (3) -- (e); 
\draw[->] (4) -- (b); 
\end{tikzpicture}
\hspace{13em}
\begin{tikzpicture}[node distance={15mm}, thick, main/.style = {}]
\node[main] (2)  {$-2$}; 
\node[main] (3) [ right of=2] {$-1$}; 
\node[main] (4) [ right of=3] {$1$}; 
\node[main] (5) [ right of=4] {$2$}; 
\node[main] (6) [ right of=5] { };

\node at (7.5, -1.5)  (h)  {${\inv(\sigma)=1,\hspace{0.5em} \ninv(\sigma)=0}$ };
\node at (7,-0)  (h)  {$\sigma=
\bigl(\begin{smallmatrix}
-2 & -1 & 1  & 2  \\
-1 & -2 & 2  & 1
\end{smallmatrix}\bigr)=\pi_1
$ };

\node[main] (b) [below  of=2] {$-2$}; 
\node[main] (c) [ below of=3] {$-1$}; 
\node[main] (d) [ below of=4] {$1$}; 
\node[main] (e) [ below of=5] {$2$}; 
\draw[->] (2) -- (c); 
\draw[->] (3) -- (b); 
\draw[->] (4) -- (e); 
\draw[->] (5) -- (d);
\end{tikzpicture}

 \vspace{0.8cm}
\begin{tikzpicture}[node distance={15mm}, thick, main/.style = {}] 
\node[main] (1) {$-3$}; 
\node[main] (2) [ right of=1] {$-2$}; 
\node[main] (3) [ right of=2] {$-1$}; 
\node[main] (4) [ right of=3] {$1$}; 
\node[main] (5) [ right of=4] {$2$}; 
\node[main] (6) [ right of=5] {$3$ };

\node[main] (a) [below  of=1] {$-3$}; 
\node[main] (b) [below  of=2] {$-2$}; 
\node[main] (c) [ below of=3] {$-1$}; 
\node[main] (d) [ below of=4] {$1$}; 
\node[main] (e) [ below of=5] {$2$}; 
\node[main] (k) [ below of=6] {$3$};
\node at (12, 0) (h) {$\sigma=
\bigl(\begin{smallmatrix}
-3 & -2 & -1 & 1  & 2  & 3\\
-2 & -3 & 1 & -1  & 3 & 2
\end{smallmatrix}\bigr)=\pi_0\pi_2
$};
\node at (12, -1.5) {${\inv(\sigma)=1,\hspace{0.5em} \ninv(\sigma)=1}$ };

\draw[->] (1) -- (b); 
\draw[->] (2) -- (a);
\draw[->] (3) -- (d);
\draw[->] (4) -- (c); 
\draw[->] (5) -- (k);
\draw[->] (6) -- (e); 

\end{tikzpicture}

\vspace{0.8cm}
\begin{tikzpicture}[node distance={15mm}, thick, main/.style = {}] 
\node[main] (1) {$-3$}; 
\node[main] (2) [ right of=1] {$-2$}; 
\node[main] (3) [ right of=2] {$-1$}; 
\node[main] (4) [ right of=3] {$1$}; 
\node[main] (5) [ right of=4] {$2$}; 
\node[main] (6) [ right of=5] {$3$ };

\node[main] (a) [below  of=1] {$-3$}; 
\node[main] (b) [below  of=2] {$-2$}; 
\node[main] (c) [ below of=3] {$-1$}; 
\node[main] (d) [ below of=4] {$1$}; 
\node[main] (e) [ below of=5] {$2$}; 
\node[main] (k) [ below of=6] {$3$};
\node at (12, 0) (h) {$\sigma=
\bigl(\begin{smallmatrix}
-3 & -2 & -1 & 1  & 2  & 3\\
-1 & -3 & 2 & -2  & 3 & 1
\end{smallmatrix}\bigr)=\pi_0\pi_2\pi_1
$};
\node at (12, -1.5) {${\inv(\sigma)=2,\hspace{0.5em} \ninv(\sigma)=1}$ };

\draw[->] (1) -- (c); 
\draw[->] (2) -- (a);
\draw[->] (3) -- (e);
\draw[->] (4) -- (b); 
\draw[->] (5) -- (k);
\draw[->] (6) -- (d); 

\end{tikzpicture}

 \vspace{0.8cm}
\begin{tikzpicture}[node distance={15mm}, thick, main/.style = {}] 
\node[main] (1) {$-3$}; 
\node[main] (2) [ right of=1] {$-2$}; 
\node[main] (3) [ right of=2] {$-1$}; 
\node[main] (4) [ right of=3] {$1$}; 
\node[main] (5) [ right of=4] {$2$}; 
\node[main] (6) [ right of=5] {$3$ };

\node[main] (a) [below  of=1] {$-3$}; 
\node[main] (b) [below  of=2] {$-2$}; 
\node[main] (c) [ below of=3] {$-1$}; 
\node[main] (d) [ below of=4] {$1$}; 
\node[main] (e) [ below of=5] {$2$}; 
\node[main] (k) [ below of=6] {$3$};
\node at (12, 0) (h) {$\sigma=
\bigl(\begin{smallmatrix}
-3 & -2 & -1 & 1  & 2  & 3\\
-1 & 2 & -3 & 3  & -2 & 1
\end{smallmatrix}\bigr)=\pi_1\pi_2\pi_0\pi_1
$};
\node at (12, -1.5) {${\inv(\sigma)=3,\hspace{0.5em} \ninv(\sigma)=1}$ };

\draw[->] (1) -- (c); 
\draw[->] (2) -- (e);
\draw[->] (3) -- (a);
\draw[->] (4) -- (k); 
\draw[->] (5) -- (b);
\draw[->] (6) -- (d); 

\end{tikzpicture}
\caption{Examples of permutations $B(2)$ and $B(3)$  with inversions.}
\label{permutacjeexmpale}
\end{figure}
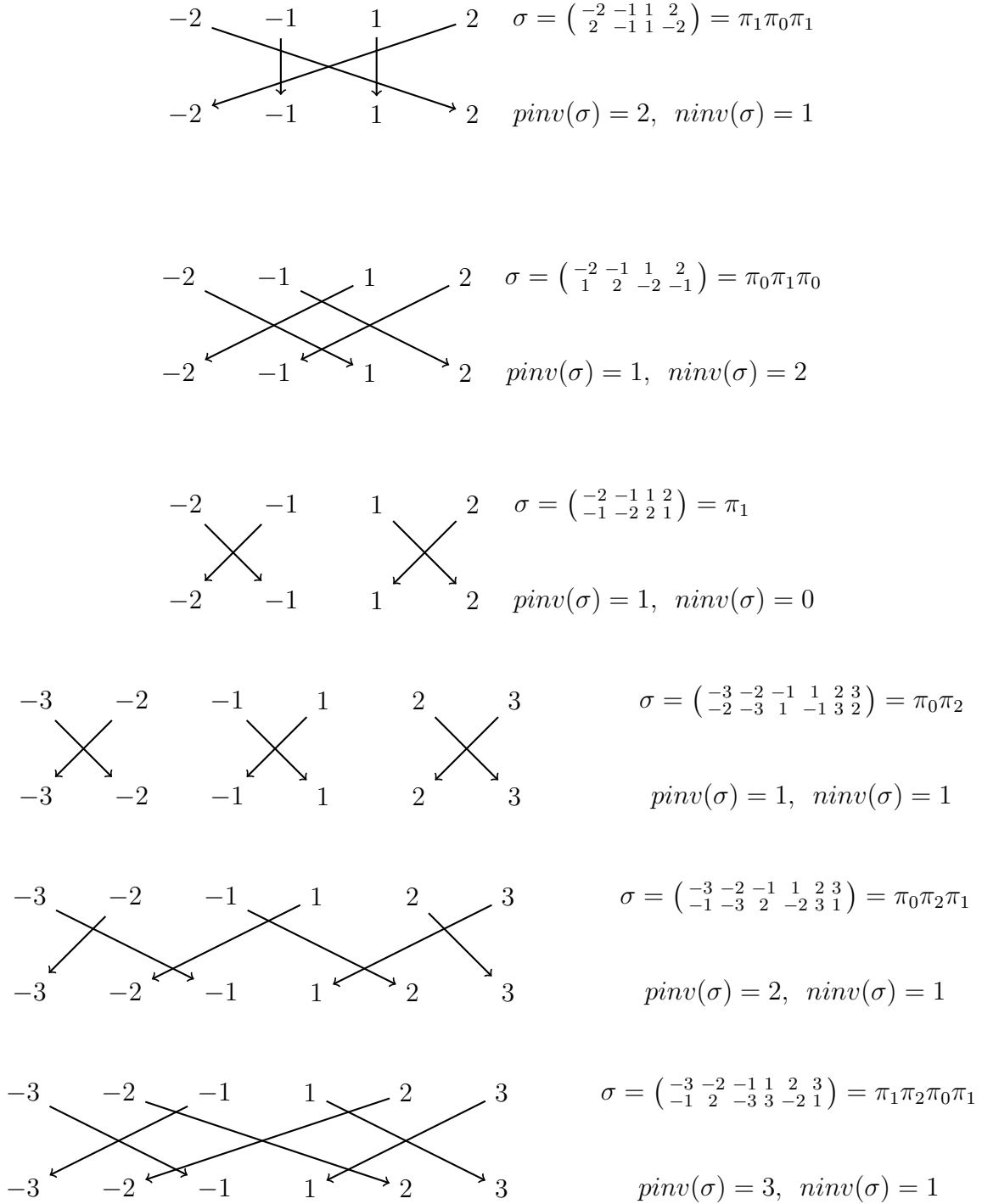
\end{remark}

\begin{definition}\label{Defi:Cumulant} Let $\pi\in \PB(n)$, $\mathbf{B}=((\bar i_m,\dots,\bar i_1), ( i_1,\dots, i_m))  \in \Pair_B(\pi)$, $\lambda_i,\lambda_{\bar i}\in \R$ and 
 $x_{\bar i}\O x_{ i} \in \HH_\R$
 and  {the B cumulant} is defined by
\begin{align*}
&\Cum \BL) \\ &:= 
\begin{cases}
\lambda_{\bar i_1} \lambda_{ i_1 }  & \text{if $\BL$ is a singleton,} \\
\langle x_{\bar i_m}, T_{{\bar i_{m-1}}} \dots T_{{\bar i_{2}}}x_{\bar i_1} \rangle 
\langle x_{ i_m}, T_{{ i_{m-1}}} \dots T_{{ i_{2}}}x_{i_1} \rangle  & \text{otherwise,}
\end{cases}
 \intertext{and their product}
&R_{\pi}^{x ,T,\lambda}:=\prod_{\BL \in \Pair_B(\pi)}\Cum {\BL}).
\end{align*}
\label{def: bcumulant}
\end{definition}

\begin{theorem}\label{lem101}
For any $i\in\{1,\ldots,n\}$  and $x_{\bar i}\otimes x_i \in \HH_\R$  we have
\begin{equation}\label{formula101}
	\begin{split}
		\state(\G(x_{\overline{ n}}\otimes x_{ n})\cdots \G(x_{\bar 1}\otimes x_{ 1}))= \sum_{\pi\in\PB(n)}  \s^{\NB(\pi)}\q^{\Cr(\pi)}  R_{\pi}^{x ,T,\lambda}.
	\end{split}
\end{equation}

\end{theorem}

\subsection{Proof of the main theorem}We begin with some special notation.

\noindent
\textbf{Extended partition.} In order to prove the main theorem, we need the set $\PB_E(n)$ of the so-called extended type-B partitions, which contains the set $\PB(n)$. We use these partitions only in the proof of Theorem \ref{twr:momentogolne}. Here each B-block consists of blocks of size at least two can be additionally marked by $E$. Additionally, each B-singleton is extended. 
\begin{definition}\label{Partycjerozszerzone}

We denote by $\PB_E(n)$ the set of partitions, which contains the set $\PB(n)$ and  such that every block of  partitions of $\pi$ is regular  or extended.
All   singletons are denoted as extended.
Let   $B\in \pi$ and  $\bar B\in \pi$ be a reflection block, then both must be denoted as regular or extended as $$(\bar b,\dots,\bar a),( a,\dots, b) \text{ or }(\bar b,\dots,\bar a)_E,( a,\dots, b)_E.$$ 
\end{definition}

\begin{example}For example,
\begin{align*} \PB_E(2)=&\Big\{\big\{(\bar 2)\primes,(2)\primes,(\bar 1)\primes,(1)\primes\big\}, \big\{(\bar 2,\bar 1),(1,2)\big\}, \big\{(\bar 2,1),(\bar 1,2)\big\}, \big\{(\bar 2,\bar 1)\primes,(1,2)\primes\big\},\\& \big\{(\bar 2,1)\primes,(\bar 1,2)\primes\big\}\Big\}.\end{align*}
\end{example}

\noindent 
For $\pi \in \PB_E(n) $ we denote by ${\BlockE}(\pi)$ the set of pairs of blocks of $\pi$ that are marked by $E$.

\noindent \textbf{\emph{Cover extreme
elements of extended blocks}}. We also need the new statistic to represent the extreme elements covered in the extended blocks. Let 
$\MaxC(\pi)$ for $\pi \in \PB_E(n)$ be the number of \emph{cover extreme elements of extended blocks}
 namely, we define
\begin{align*}
&\MaxC(\pi)=\#\big\{ \bigl(((\bar a_k,\cdots, \bar a_1)\primes,(a_1,\cdots, a_k)\primes),W\bigr)\in {\BlockE}(\pi)\times \Semi_{B}(\pi) \mid 
\\& i< a_k<j \textrm{ equivalently }  \bar j< \bar a_k<\bar i  \textrm{ for } ((\bar j,\bar i),(i,j))\in W\big\},
\end{align*}
which we will use in Theorem \ref{twr:momentogolne}, only.

\begin{example} For the partition presented in Figure \ref{RestrictedCrossings1} we see that 

\begin{itemize}
\item if $\pi=\{(-10,-6,5),(-5,6,10),(-9,-8),(8,9),(-7,-4,-3,1),(-1,3,4,7),\\(-2)\primes,(2)\primes\}$, then $\rc{(\pi)}=3$, $\NB(\pi)=2$, $\InS(\pi)=3$ and $\MaxC(\pi)=3$;
\item if $\pi=\{(-10,-6,5)\primes,(-5,6,10)\primes,(-9,-8),(8,9),(-7,-4,-3,1),(-1,3,4,7),\\(-2)\primes,(2)\primes\}$, then $\rc{(\pi)}=3$, $\NB(\pi)=2$, $\InS(\pi)=3$ and $\MaxC(\pi)=3$; 
\item if $\pi=\{(-10,-6,5),(-5,6,10),(-9,-8)\primes,(8,9)\primes,(-7,-4,-3,1),(-1,3,4,7),\\(-2)\primes,(2)\primes\}$, then $\rc{(\pi)}=3$, $\NB(\pi)=2$, $\InS(\pi)=3$ and $\MaxC(\pi)=4$; 
\item if $\pi=\{(-10,-6,5),(-5,6,10),(-9,-8),(8,9),(-7,-4,-3,1)\primes,(-1,3,4,7)\primes,\\(-2)\primes,(2)\primes\}$, then $\rc{(\pi)}=3$, $\NB(\pi)=2$, $\InS(\pi)=3$ and $\MaxC(\pi)=4$;
\item if $\pi=\{(-10,-6,5),(-5,6,10),(-9,-8)\primes,(8,9)\primes,(-7,-4,-3,1)\primes,(-1,3,4,7)\primes,\\(-2)\primes,(2)\primes\}$, then $\rc{(\pi)}=3$, $\NB(\pi)=2$, $\InS(\pi)=3$ and $\MaxC(\pi)=5$. 
\end{itemize}

\begin{figure}[ht]
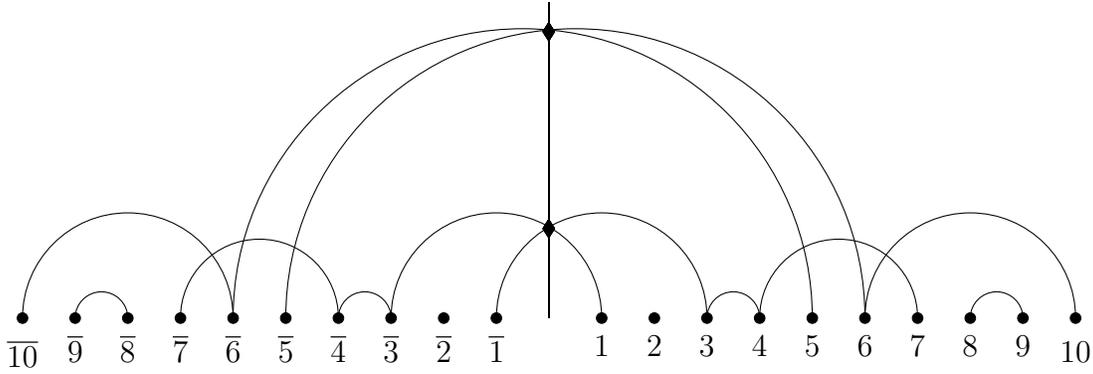

	\begin{center}
		\MatchingMeandersnew{10}{-10/-6, -6/5, -9/-8, 8/9, -5/6, 6/10, -7/-4, -4/-3, -3/1, -1/3, 3/4, 4/7}
	\end{center}  
\caption{The example of partition $\pi \in \P^{B}(10).$}
 \label{RestrictedCrossings1}
\end{figure} 
\end{example}
In order to simplify  notation, we define the following operators i.e. $\OptTTyl_{B_E}=T_{{i_m}} \dots T_{{i_{3}}}T_{{i_{2}}}$ and $\OptTTyl_{\bar B_E}=T_{{\bar i_m}} \dots T_{{\bar i_{3}}}T_{{\bar i_{2}}} $, which map  $H$  into $H$ and which are indexed by the extended block $$(\bar B_E, B_E)=((\underbracket{\bar i_m}_{=\r_{\bar B_E}},\dots,\underbracket{\bar i_1}_{=\l_{\bar B_E}})_E,(\underbracket{i_1}_{=l_{B_E}},\dots,\underbracket{i_m}_{=\r_{B_E}})_E).$$ 
For this B-block we also denote the left and right legs.  
For the regular block $$(\bar B, B)=((\underbracket{\bar i_m}_{=\r_{\bar B}},\dots,\underbracket{\bar i_1}_{=\l_{\bar B}}),(\underbracket{i_1}_{=l_B},\dots,\underbracket{i_m}_{=\r_{B}})).$$  we denote $\OptT_{B}=T_{{i_{m-1}}}T_{{i_{m-2}}} \dots T_{{i_2}}$ and $\OptT_{\bar B}=T_{{\bar i_{m-1}}} T_{{\bar i_{m-2}}} \dots T_{{\bar i_{2}}}$. 

With the notation above we also introduce: 
\begin{align*}
\begin{split}
&\Cumm=\prod_{\substack{ (\bar B,B) \in \Pair_B(\pi)}} \langle x_{\r_{\bar B}}, \OptT_{\bar B}x_{\l_{ \bar B}}\rangle 
 \langle x_{\r_B} ,\OptT_{B}x_{\l_B}\rangle 
 \\ 
&\CummTensor=\Big[\bigotimes_{\substack{\bar B_E \in \pi}}\{ \OptTTyl_{\bar B_E} x_{\l_{ \bar B_E}}\}_{\min \bar B_E}\Big]
\otimes
\Big[\bigotimes_{\substack{ B_E \in \pi}}\{ \OptTTyl_{B_E} x_{\l_{ B_E}}\}_{\min B_E}\Big].
\end{split}
\end{align*}
Notice that in the above formula we use the following bracket notation $\{\cdot\}_{\min B_E}$
, which should
be understood that the position of $\cdot$ (in the tensor product) is ordered with respect to the
$\min B_E$.

\begin{example} 
For example, 
\begin{itemize}
\item if $\pi=\{(-4)\primes,(4)\primes,(-6)\primes,(6)\primes,(-7,-5,-2)\primes,(2,5,7)\primes,(-3,-1)\primes,(1,3)\primes\}$, then $$\CummTensor=T_{{-7}}T_{{-5}}x_{-2} \otimes x_{-6} \otimes x_{-4} \otimes T_{{-3}}x_{-1}\otimes T_{3}x_1 \otimes x_4 \otimes x_6 \otimes T_{7}T_{5}x_2;$$
\item if $\pi=\{(-4)\primes,(4)\primes,(-6)\primes,(6)\primes,(-7,-5,-2)\primes,(2,5,7)\primes,(-3,1)\primes,(-1,3)\primes\}$, then $$\CummTensor=T_{-7}T_{-5}x_{-2} \otimes x_{-6} \otimes x_{-4} \otimes T_{-3}x_{1}\otimes T_{3}x_{-1} \otimes x_4 \otimes x_6 \otimes T_{7}T_{5}x_2.$$
\end{itemize}
\end{example}
We also use the following conventions for $\epsilon\in\{1,\ast,E\}$
\[
\B^{\epsilon}(x_{\bar i}\otimes x_i) = 
\begin{cases}
\B^{*}(x_{\bar i}\otimes x_i) & \text{ if } \epsilon=\ast, \\
p_q(x_{\bar i}\otimes x_i)+ \alpha \ell_{q}({x_{\bar i}\otimes x_i}) & \text{ if } \epsilon=1,
\\
\r_q^{T_{\bar i} \otimes   T_i} + \alpha \ell_{q}^{N, T_{\bar i}  \otimes T_i}& \text{ if } \epsilon=E.
\end{cases}
\]
\noindent

\begin{theorem}\label{twr:momentogolne}
For any $i\in\{1,\ldots,n\}$, $x_{\bar i}\otimes x_i \in \HH_\R$ and any $\epsilon=(\epsilon(1), \dots, \epsilon(n))\in\{1,\ast,E\}^n$, we have
\begin{equation}\label{formula101}
\B^{\epsilon(n)}(x_{\bar n}\otimes x_n)\cdots \B^{\epsilon(1)}(x_{\bar 1}\otimes x_1)\Omega \otimes \Omega = \sum_{\pi\in\PB_{E;\epsilon}(n)} \alpha^{\NB(\pi)}q^{\rc(\pi) +\MaxC(\pi)} \Cumm\CummTensor.
\end{equation}
\end{theorem}

\begin{proof}
[Proof of Theorem \ref{twr:momentogolne}]
Given
$\epsilon=(\epsilon(1), \dots, \epsilon(n))\in\{1,\ast, E\}^{n}$, let
$\P_{E;\epsilon}^{B}(n)$ be the set of partitions $\pi\in
\P_{E}^{B}(n)$ such that
\begin{itemize}
 \item if $((\bar a_k ,\dots,\bar a_1) , (a_1,\dots, a_k))$ is a B-block of
$\BlockE(\pi)$ then  $$\epsilon(|a_1|)= \ast,\epsilon(|a_2|)=E,\epsilon(|a_3|)= E,\dots,\epsilon(|a_{k-1}|)=E,\epsilon(|a_k|)=E;$$
 \item if $((\bar a_k ,\dots,\bar a_1) , (a_1,\dots, a_k))$ is a B-block of
	$\Pair_B(\pi)$ then  $$\epsilon(|a_1|)= \ast,\epsilon(|a_2|)=E,\epsilon(|a_3|)= E,\dots,\epsilon(|a_{k-1}|)=E,\epsilon(|a_k|)=1.$$
\end{itemize}
 We denote by $(A,B)$
the concatenation of the
block $A\in \pi$ 
and $B\in \pi$. For example, if $A=(-1,2)$ and $B=(4,5)$, then $(A,B)=(-1,2,4,5)$. 

    Observe that, when $n=1$, then $\B({x_{\bar 1}\O x_{ 1}})\Omega\O  \Omega=p(T_{\bar 1}\O T_{{1}})\Omega\O  \Omega=0$  and $$\B^\ast({x_{\bar 1}\O x_{1}})\Omega\O \Omega={x_{\bar 1}\O x_{ 1}}.$$  
Suppose that  $x_{\bar i}\O x_{ i} \in  \HH_{\R}$, $i\in\{1,\dots,n-1\}$
and any $\epsilon=(\epsilon(1), \dots, \epsilon(n-1))\in\{1,\ast,E\}^n$,  we have
\begin{align*}
\begin{split}
\B^{\epsilon(n-1)}(x_{\overline{n-1}}\O x_{n-1})\cdots \B^{\epsilon(1)}(x_{\bar 1}\O x_{1})\Omega \O \Omega =& \sum_{\pi\in\PB_{E;\epsilon}(n-1)}   \alpha^{\NB(\pi)}q^{\rc(\pi) +\MaxC(\pi)} \Cumm\CummTensor.
\end{split}
\end{align*}
We will show that the action of $\B^{\epsilon(n)}(x_{\bar n}\O x_{ n})$ corresponds to the inductive graphic description of set tensor partitions.
We fix $\pi\in \PB_{E;\epsilon}(n-1)$ and run the argument below over all partitions of this type. Suppose that the set  $$ \BlockE(\pi)=\{(\bar S_E^1,S_E^1),\dots,(\bar S_E^i,S_E^i),\dots,  (\bar S_E^r,S_E^r)\}$$ has pair of blocks  on the positions $\bar s_{ r},\dots,  \bar s_{ 1},s_1,\dots,s_r$ in the following sense:
$$\begin{array}{ccccccccccccccccccc}
\bar s_{ r} & < \cdots  <& \bar s_{ i} &<  \cdots  <& \bar s_{ 1}&<  &s_{1}&<\cdots <& s_{i} &<  \cdots <& s_r
\\ \downarrow &  & \downarrow& &\downarrow&   & \downarrow & & \downarrow & & \downarrow
\\ \scriptstyle{\min \bar S^{ r}_E }& < \cdots <& \scriptstyle{\min \bar S^{ i}_E }&< \cdots  <&\scriptstyle{ \min \bar S^{ 1}_E}&<  & \scriptstyle{\max  S^{ 1}_E} &<\cdots <& \scriptstyle{\max  S^{ i}_E} &<  \cdots <& \scriptstyle{\max  S^{ r}_E }.
\end{array}$$
Suppose that a partition $\pi$ has blocks $( \bar S_E,  S_E)=(\bar S_E^i,S_E^i)\in \BlockE(\pi)$ on the $(\bar s_{ i},s_i)$ position, i.e. $(\bar s_{ i},s_i)=(\min \bar S_E,\max  S_E)$. 
This block  have the following contribution to $ \CummTensor$: 
\begin{align*}
  \{\cdot\}_{\bar s_{
 r}}\otimes\dots\otimes \{\OptTTyl_{\bar S_E}x_{\l_{ \bar S_E}}\}_{\bar s_{ i}}\otimes \dots\otimes\{\cdot\}_{\bar s_{ 1}}\O\{\cdot\}_{s_{ 1}}\otimes\dots\otimes \{\OptTTyl_{S_E}x_{\l_{ S_E}}\}_{s_{i}}\otimes \dots\otimes\{\cdot\}_{s_r}
\end{align*}
Let    $W_{1}=((\bar b_1,\bar a_1),(a_1,b_1)),\dots, W_{u}=((\bar b_u,\bar a_u),(a_u,b_u))$ be the B-arcs, which cover  $(\bar s_{ i},s_{i})$ in a sense that $a_j< s_i <b_j$ (equivalently $\bar b_j< \bar s_{ i} < \bar a_j $) for $j\in\{1,\dots,u\}$. 

\noindent Case 1. If $\epsilon(n)=\ast$, then the operator $\B^\ast(x_{\bar n}\O x_{ n})$ acts on the tensor product by adding (extended) singleton on the left and right. 
This operation pictorially corresponds to adding the singletons  as follows
\begin{align*}
 \{\cdot\}_{\bar s_{
 r}}\otimes\dots\otimes \{\cdot\}_{s_r}   \mapsto\{
(\bar n)\}_{{
\bar n}} \otimes \{\cdot\}_{\bar s_{
 r}}\otimes\dots\otimes \{\cdot\}_{s_r} \otimes  \{(n)\}_{{
 n}} 
\end{align*}


This yields a new partition $\tilde\pi\in \PB_{E;\epsilon}(n)$. 
This map $\pi\mapsto \tilde \pi$ 
does not change the numbers related to our statistic  $ \Cr , \MaxC $ and $\NB$, which is compatible with the fact that the creation action
does not change the coefficient, hence the formula \eqref{formula101} holds when we moved from  $n-1$ to $n$ and $\epsilon(n)=\ast$.

If $\epsilon(n)=1$, then we have two case correspond to \cref{eq:case2proofcum} and \cref{eq:case2proofcum1}. 
Case 2 a). If  $p_q(x_{\bar n}\O x_{ n})$ acts on the tensor product, then new $r$ terms appear. In terms $\bar s_{ i}$ and $s_{i}$ the inner product 
\begin{align}\label{eq:case2proofcum}
\langle x_{\bar n}, \OptTTyl_{\bar S_E}x_{\l_{\bar S_E}} \rangle 
\langle x_{ n }, \OptTTyl_{ S_E}x_{\l_{ S_E}} \rangle
\end{align}
appears with coefficient $q^{r-i+1}$. Graphically this corresponds to getting a set partition $\tilde{\pi} \in \PB_{E;\epsilon}(n)$ by adding $n$ and $\bar n$ to $\pi$ and creating a new regular block  $(\bar n,\bar S)$ and
$(S,n)$ by adding the arcs $(\bar n,\bar s_{ i})$ to $\bar S_E$  and the second  $(s_{i}, n)$ to $ S_E$ with $\e(\bar n)=\e( n)=1$.
 We see that   \cref{eq:case2proofcum} can be written in the form:   
$$\langle x_{\r_{(\bar n,\bar S)}} ,\OptT_{(\bar n,\bar S)}x_{\l_{(\bar n,\bar S)}}\rangle \langle x_{\r_{( S,n)} }, \OptT_{( S,n)}x_{\l_{ ( S,n)}}\rangle.$$
The new B-arc $\BL=((\bar n,\bar s_{ i}),(s_i,n))$ crosses the pair of B-arcs $(\overline{W}_1,W_1)\dots, (\overline{W}_u,W_u)$ and so increases the number of crossings by $u$,  but decreases the number of $\MaxC(\pi)$ by $u$. 
Also new B-arc covers  extended block on positions   $s_{i+1}, \dots, s_{r}$ ($=\bar s_{r}, \dots,\bar  s_{i+1}$). 
Altogether we have:
\begin{align*}  \begin{split}
&\Cr(\tilde{\pi})=\Cr(\pi)+u , \\
&\MaxC(\tilde{\pi})=\MaxC(\pi)-u+r-i, \\
&\NB(\tilde{\pi})=\NB(\pi)
  \end{split} 
  \end{align*}
So we see that the exponent of $q$'s,    increases by $q^{r-(i-1)}$,  see \cref{figcase2} (a). 
\begin{figure}[h]
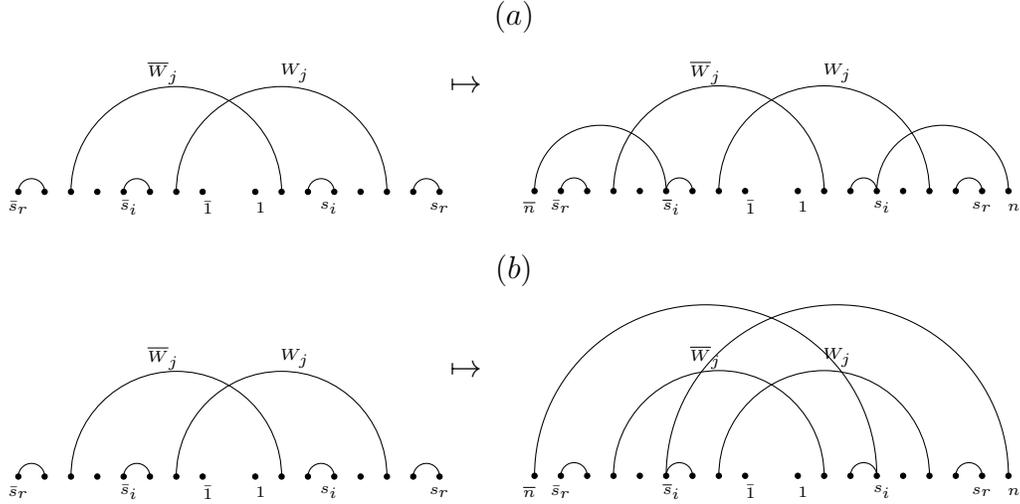

\begin{center}
$(a)$
\end{center}%
\begin{center}
	\MatchingMeandersPositiveTracea{8}{-4/-3,3/4,-6/2,-2/6,-8/-7, 7/8}{} 
	\MatchingMeandersPositiveTraceb{9}{-4/-3,3/4,-6/2,-2/6,-8/-7, 7/8,-9/-4,4/9}{} 
\end{center}  
\begin{center}
$(b)$
\end{center}%
\begin{center}
\MatchingMeandersPositiveTracea{8}{-4/-3,3/4,-6/2,-2/6,-8/-7, 7/8 }{} 
\MatchingMeandersPositiveTraceb{9}{-4/-3,3/4,-6/2,-2/6,-8/-7, 7/8,-9/4,-4/9}{} 
\end{center} 

\caption{The visualization of the actions of $p_\q(x_{\overline{n}}\otimes  x_{n})$ (a) and $\alpha \ell_{q}(x_{\bar n}\O x_{ n})$ (b).}
\label{figcase2}
\end{figure}

Case 2 b). If $\alpha \ell_{q}(x_{\bar n}\O x_{ n})$ acts on the tensor product, then new $r$ terms appear. In $(\bar s_i,s_i)^{th}$ term  the inner product 
\begin{align}\label{eq:case2proofcum1}
\langle x_{\bar n }, \OptTTyl_{S_E}x_{\l_{S_E}} \rangle
\langle x_{n}, \OptTTyl_{\bar S_E}x_{\l_{ \bar S_E}} \rangle
\end{align}
appears with coefficient $\alpha q^{r+i-1}$. 
Graphically this corresponds to getting a set partition $\tilde{\pi} \in \PB_{E;\epsilon}(n)$ by adding $n$ and $\bar n$ to $\pi$ and creating a new regular block 
$(\bar S,n)$ and $(\bar n, S)$  by adding the arcs $(\bar s_{ i},n)$ to $\bar S_E$  and the second  $(\bar n, s_{i})$ to $ S_E$, with $\e(\bar n)=\e( n)=1$.
 We see that   \cref{eq:case2proofcum1} can be written in the form:   
$$\langle x_{\r_{ ( \bar n,S)} }, \OptT_{( \bar n,S)}x_{\l_{( \bar n,S)}}\rangle \langle x_{\r_{( \bar S,n)}} ,\OptT_{(\bar S,n)}x_{\l_{(\bar S, n)}}\rangle .$$
The new negative B-arc $\BL=((\bar n,s_{i}),(\bar s_{ i},n))$ crosses the B-arcs $W_1, \dots, W_{u}$ and so increases the number of crossings by $u$,  but decreases the number of $\MaxC(\pi)$ by $u$. 
Now some new inner extended block on position   $s_{\overline{i-1}},\dots, s_{\bar 1}, s_1,\dots  \check{s_{i}}, \dots, s_{r}$  appear. 
Altogether we have:
\begin{align*}  \begin{split}
&\Cr(\tilde{\pi})=\Cr(\pi)+u , \\
&\MaxC(\tilde{\pi})=\MaxC(\pi)-u+r-1+(i-1), \\
&\NB(\tilde{\pi})=\NB(\pi)+1
  \end{split} 
  \end{align*}
So we see that the exponent  of $\alpha$ increases by $1$
and the exponent of of $q$'s,   increases by $q^{r-1+(i-1)}$,  see Figure \ref{figcase2} (b).

Case 3. If $\epsilon(n)=E$, then we have two cases corresponding to $\r_q^{T_{\bar n}\otimes T_n}$ and $\alpha \ell_{q}^{N,T_{\bar n}\otimes T_n}$. If  we use the equation \eqref{rqT} and  \eqref{lqT}, then in the $(\bar s_{ i},s_i)^{\rm th}$ term   we delete the elements $$\OptTTyl_{S_E}x_{\l_{ S_E}}\text{ and } \OptTTyl_{\bar S_E}x_{\l_{ \bar S_E}}\text{ from }\CummTensor,$$   
and move them on the \emph{edge} positions  as shown in  Figure \ref{fig:FiguraExemple4}, which finally contribute to a new component $\widehat{\mathrm{ R}}^\mb{x}_{\tilde\pi}$ . 
The new coefficients are now also appearing  $q^{r-i}$ and 
 $\alpha q^{r+i-2}$ in the action  \eqref{rqT} and  \eqref{lqT} respectively. 

 \begin{figure}[h]
\begin{center}
  \begin{tikzpicture}[thick,font=\small,scale=.9]
     \path 
           (-10,0) node[] (b) {\tiny{$\widehat{\mathrm{ R}}^\mb{x}_{\tilde\pi}=$}}
           (-8,0) node[] (bc) {\tiny{$T_{\bar n}(\OptTTyl_{\bar S_E}x_{\l_{\bar S_E}})\otimes$}}
           (-8,0.4) node[] (pom1) {$\blacktriangledown$}
           (-6.3,0) node[] (d) {\tiny{$...$}}
           (-4.5,0) node[] (g) {\tiny{$\otimes\{\xcancel{\OptTTyl_{\bar S_E}x_{\l_{\bar S_E}}}\}_{\bar s_{ i}}\otimes$}}
           (-2.75,0) node[] (gg) {\tiny{$...$}}
           (-2,0) node[] (gg) {\tiny{$\otimes\{\cdot\}_{s_{\bar 1}}$}}
             (-1,2) node[] (wed) {$\r_q^{T_{\bar n}\otimes T_n}$}
           (-1,0) node[] (gg) {$\O$};
     \draw[thick] (bc) -- +(0,1) -| (g);
   \hspace{-0.9cm}
     \path 
           (1,0) node[] (bc) {\tiny{$\{\cdot\}_{s_{1}}\otimes$}} 
           (6.9,0.4) node[] (pom1) {$\blacktriangledown$}
           (1.72,0) node[] (d) {\tiny{$...$}}
           (3.4,0) node[] (g) {\tiny{$\otimes \{\xcancel{\OptTTyl_{S_E}x_{\l_{ S_E}}}\}_{s_{i}}\otimes$}}
      (5.2,0) node[] (gg) {\tiny{$...$}}
           (6.9,0) node[] (gg) {\tiny{$ \otimes T_{{n}}(\OptTTyl_{ S_E}x_{\l_{  S_E}})$}} ;
     \draw[thick] (g) -- +(0,1) -| (gg);
   \end{tikzpicture}
   \end{center}  
\begin{center}
  \begin{tikzpicture}[thick,font=\small,scale=.9]
     \path 
           (-10,0) node[] (b) {\tiny{$\widehat{\mathrm{ R}}^\mb{x}_{\tilde\pi}=$}}
           (-8,0) node[] (bc) {\tiny{$T_{\bar n}(\OptTTyl_{ S_E}x_{\l_{S_E}})\otimes$}}
           (-8,0.4) node[] (pom1) {$\blacktriangledown$}
           (-6.3,0) node[] (d) {\tiny{$...$}}
           (-4.5,0) node[] (g) {\tiny{$\otimes\{\xcancel{\OptTTyl_{\bar S_E}x_{\l_{\bar S_E}}}\}_{\bar s_{ i}}\otimes$}}
           (-2.75,0) node[] (gg) {\tiny{$...$}}
           (-2,0) node[] (gg) {\tiny{$\otimes\{\cdot\}_{s_{\bar 1}}$}}
           (-1,0) node[] (gg) {$\O$}
           (0,0) node[] (bcb) {\tiny{$\{\cdot\}_{s_{1}}\otimes$}} 
           (5.9,0.4) node[] (pom1) {$\blacktriangledown$}
           (0.72,0) node[] (d) {\tiny{$...$}}
           (2.4,0) node[] (gd) {\tiny{$\otimes \{\xcancel{\OptTTyl_{S_E}x_{\l_{ S_E}}}\}_{s_{i}}\otimes$}}
      (4.2,0) node[] (gg) {\tiny{$...$}}
           (-1,2.5) node[] (wed) {$\ell_{q}^{N,T_{\bar n}\otimes T_n}$}
           (5.9,0) node[] (ggg) {\tiny{$ \otimes T_{{n}}(\OptTTyl_{\bar S_E}x_{\l_{\bar  S_E}})$}};
       \draw[thick] (bc) -- +(0,1) -| (gd);
       \draw[thick] (g) -- +(0,1.5) -| (ggg);
   \end{tikzpicture}
   \end{center}  
 \caption{The visualization of the action $\r_q^{T_{\bar n}\otimes T_n}$ and $\ell_{q}^{N,T_{\bar n}\otimes T_n}$ on the 
 tensor product $\CummTensor$. }
\label{fig:FiguraExemple4}
\end{figure}
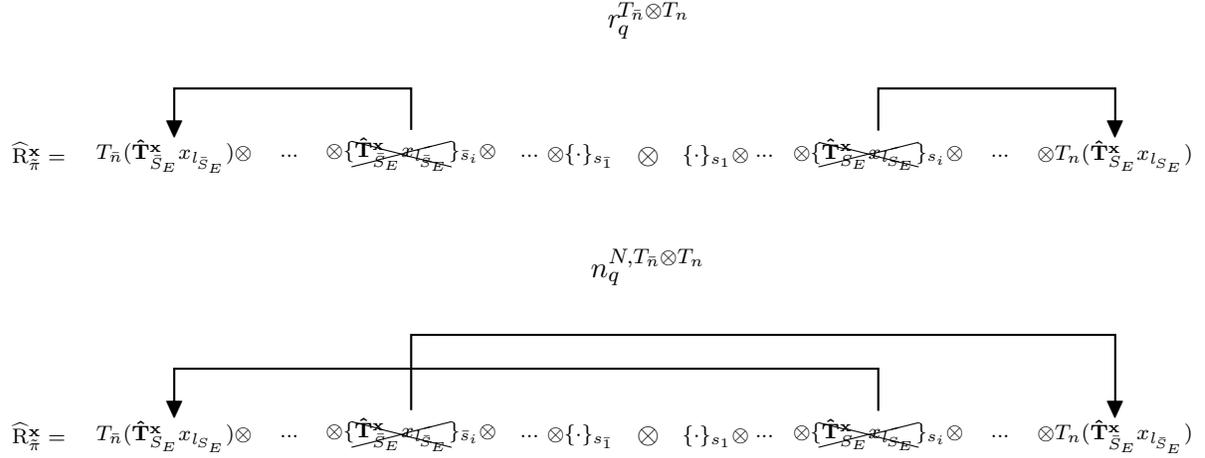

In the action  \eqref{rqT}  we get a new partition $\tilde{\pi} \in \PB_{E;\epsilon}(n)$ by adding $n$ to $S_E$, $\bar n$ to $\bar S_E$ (with the first arc $(\bar n,\bar s_{ i})$  and $(s_i,n)$  and creating two blocks in $ \BlockE(\tilde{\pi})$ with $\e(\bar n)=\e( n)=E$. Now the  extreme elements  of newly created blocks are $n$ and $\bar n$ and  so we can calculate the change in the statistic generated by the new arc in Case 2a, because new arcs cannot be covered of some  arc. 
This situation is also compatible with changes inside the tensor product, i.e.
$$\OptTTyl_{(S,n)_E}x_{\l_{ S_E}}=T_{{n}}(\Tens_{S_E} x_{\l_{ S_E}})\text{ and }\OptTTyl_{(\bar n,\bar S)_E}x_{\l_{ \bar S_E}}=T_{{\bar n}}(\Tens_{\bar S_E} x_{\l_ {\bar S_E}}).$$

We proceed similarly in the case of the formula \eqref{lqT} and create the new extended block by adding $n$ to $\bar S_E$, $\bar n$ to $ S_E$, which contribute to the new tensor product
$$\OptTTyl_{(\bar n,S)_E}x_{\l_{ S_E}}=T_{{\bar n}}(\Tens_{S_E} x_{\l_{ S_E}})\text{ and }\OptTTyl_{(\bar S,n)_E}x_{\l_{ \bar S_E}}=T_{{ n}}(\Tens_{\bar S_E} x_{\l_ {\bar S_E}}).$$

Note that as $\pi$ runs over $\pi\in \PB_{E;\epsilon}(n-1)$, every set partition $\tilde{\pi}\in \PB_{E;\epsilon}(n)$  appears exactly once either in   Case 1, 2 or in Case 3 as shown by induction that the formula \eqref{formula101} holds for all $n \in \N$. 
\end{proof}
\begin{proof}
[Proof of Theorem \ref{lem101}]
We now present the final step. First, let us notice that for $\epsilon\in\{1,\ast,E\}^n$ we have
\begin{align} \label{wick:prawieglowna}
&\state\big(\B^{\epsilon(n)}(x_{\bar n}\O x_{ n}) \cdots \B^{\epsilon(1)}(x_{\bar 1}\O x_{1}) \big)=
\sum_{ \substack{\pi\in\PB_{E;\epsilon}( n)\\ \BlockE(\pi)=\emptyset
}}   
\alpha^{\NB(\pi)}q^{\rc(\pi) } \Cumm\end{align}
Indeed, from equation \eqref{formula101} we see that the following condition must hold: $\CummTensor=\Omega\O  \Omega$. This will happen if and only if $\BlockE(\pi) =\emptyset,$ which implies \eqref{wick:prawieglowna}. If  $\BlockE(\pi)=\emptyset$, then $\Cumm=R_{\pi}^{x,T,0}$  so by taking the sum over all $\epsilon$ from equation \eqref{wick:prawieglowna}, we see that
\begin{align}\label{eq:rownaiemomentysrzero}
\begin{split}
\state&\Big( \big(\G^{\lambda_{\bar n}\otimes \lambda_n}(x_{\overline{ n}}\otimes x_{ n})-\lambda_{\bar n} \O \lambda_{ n}\big) \cdots  \big(\G^{\lambda_{\bar 1}\otimes \lambda_1}(x_{\overline{ 1}}\otimes x_{ 1})-\lambda_{\bar 1}\O\lambda_{ 1} \big)\Big)\\&=
\state(\G(x_{\overline{ n}}\otimes x_{ n})\cdots \G(x_{\bar 1}\otimes x_{ 1}))\\&=\sum_{\e\in\{1,\ast,E\}^{n}} \state( \B^{\epsilon(n)}(x_{\overline{n}}\otimes  x_{n})\cdots \B^{\epsilon(1)}(x_{\overline{1}}\otimes  x_1)) \\&=\sum_{\e\in\{1,\ast, E\}^{n}}  \sum_{ \substack{\pi\in\PB_{E;\epsilon}( n)\\ \BlockE(\pi)=\emptyset
}}  \s^{\NB(\pi)}q^{\Cr(\pi)} 
R_{\pi}^{x,T,0} 
\\&=  \sum_{ \substack{\pi\in\PB( n)
\\ 
\Sing(\pi)=\emptyset 
}}  \s^{\NB(\pi)}q^{\Cr(\pi)} 
R_{\pi}^{x,T,0} 
\end{split}
\end{align}
Where in the above formula we used  the following fact
\begin{align*}
&\bigsqcup_{\e\in\{1,\ast,E\}^{n}} \{ \pi\in \PB_{E,\e} (n) \mid \BlockE(\pi)=\emptyset\}=\\&\bigsqcup_{\e\in\{1,\ast,E\}^{n}} \{\pi \in \P_{E,\e}^{B}(n)\mid (\bar a_k,\dots,\bar a_1 ), (a_1,\dots,a_k)\in \pi  \implies \\& \text{ $\epsilon(|a_1|)= \ast,\epsilon(|a_{2}|)=E, 
\dots  , \epsilon(|a_{k-1}|)= E,\epsilon(|a_k|)=1$}\}
\}
\\&=\P_{\geq 2}^B(n).
\end{align*}

We also see that
\begin{align*} 
&\state\Big( \big(\G^{\lambda_{\bar n}\otimes \lambda_n}(x_{\overline{ n}}\otimes x_{ n})-\lambda_{\bar n}\O\lambda_{ n} +\lambda_{\bar n}\O\lambda_{n} \big) \cdots \big(\G^{\lambda_{\bar 1}\otimes \lambda_1}(x_{\overline{ 1}}\otimes x_{ 1})-\lambda_{\bar 1}\O \lambda_{ 1}+\lambda_{\bar 1}\O \lambda_{ 1}\big)\Big)
\intertext{by equation \eqref{eq:rownaiemomentysrzero}, we get}
&=\sum_{V \subset \{1,\dots,n\}} \Bigg[\prod_{i\in V} \lambda_i\lambda_{\bar i} \sum_{\pi \in \PB_{\geq 2}([n]\setminus V)}  \alpha^{\NB(\pi)}q^{\rc(\pi) }  R_{\pi}^{x,T,0} \Bigg]
\\&= \sum_{\pi \in \PB(n)}  \alpha^{\NB(\pi)}q^{\rc(\pi) }  R_{\pi}^{x,T,\lambda}=\state \big(\G^{\lambda_{\bar n}\otimes \lambda_n}(x_{\overline{ n}}\otimes x_{ n}) \cdots \G^{\lambda_{\bar 1}\otimes \lambda_1}(x_{\overline{ 1}}\otimes x_{ 1})\big)
\end{align*} 

\end{proof}

\begin{example} \label{exkumulanty}The   set of  partitions of type B for $$\state(\G(x_{\overline{ 4}}\otimes x_{ 4})\cdots \G(x_{\bar 1}\otimes x_{ 1}))$$ 
can be graphically represented as shown in  \cref{figexample1}.
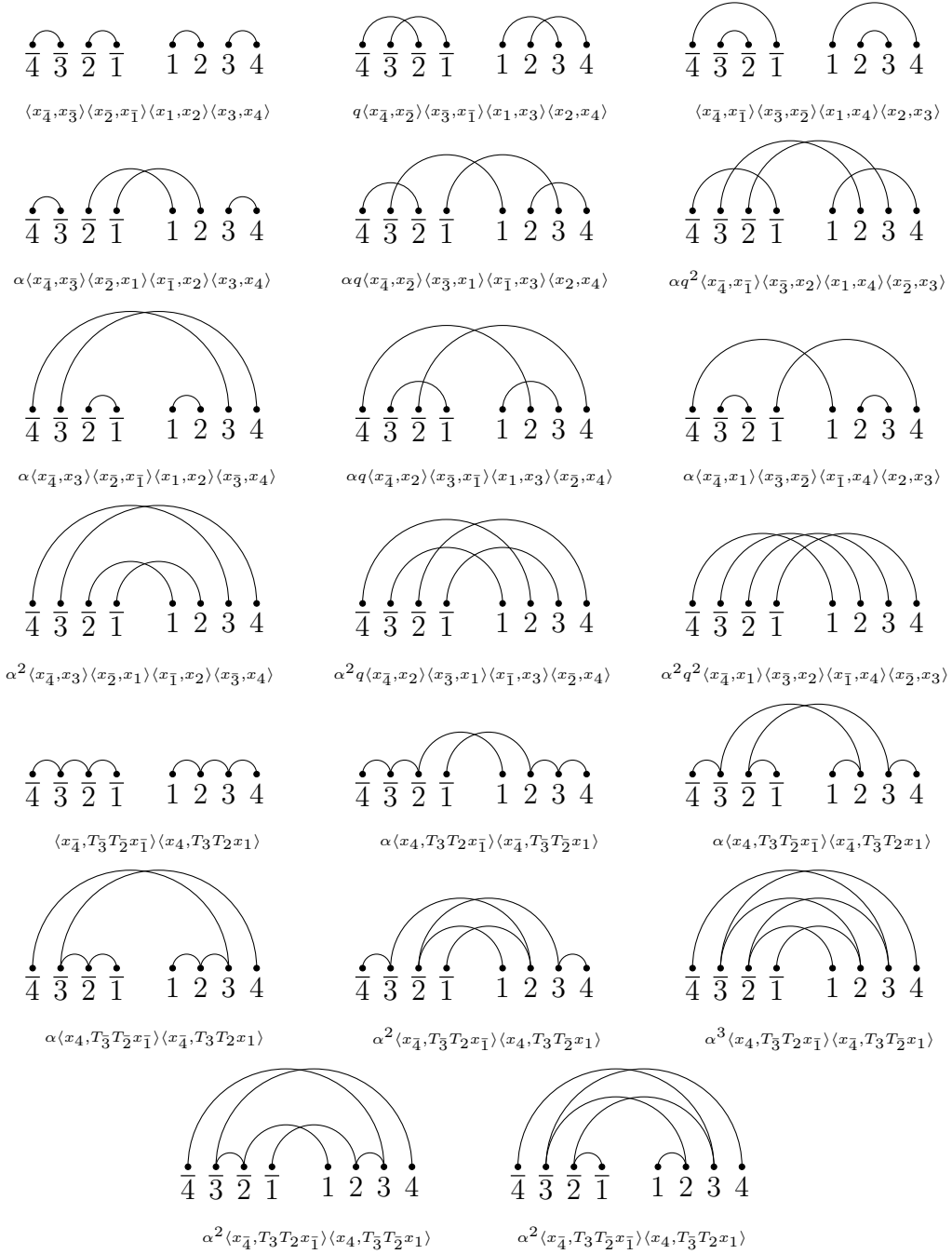
\begin{figure}[htp]

\begin{center}
\MatchingMeandersab{4}{-4/-3, -2/-1,1/2,3/4 }{} \hspace{1.5em} \MatchingMeandersab{4}{-4/-2, -3/-1,1/3,2/4 }{} \hspace{1.5em} \MatchingMeandersab{4}{-4/-1, -3/-2,2/3,1/4 }{} 
\end{center} 
\begin{tikzpicture}[scale=0.5]
\node at (0,1) {};
\node at (4,1) {$ \scriptscriptstyle{\langle x_{\bar 4}, x_{\bar 3}\rangle \langle x_{\bar 2}, x_{\bar 1}\rangle  \langle x_{1}, x_{2}\rangle\langle x_{3}, x_{4}\rangle }$};
\node at (13.5,1) {$\scriptscriptstyle{\q \langle x_{\bar 4}, x_{\bar 2}\rangle \langle x_{\bar 3}, x_{\bar 1}\rangle  \langle x_{1}, x_{3}\rangle\langle x_{2}, x_{4}\rangle }$};
\node at (23,1) {$\scriptscriptstyle{\ \langle x_{\bar 4}, x_{\bar 1}\rangle \langle x_{\bar 3}, x_{\bar 2}\rangle  \langle x_{1}, x_{4}\rangle\langle x_{2}, x_{3}\rangle }$};
\end{tikzpicture} 
\begin{center} 
\MatchingMeandersab{4}{-4/-3, -2/1,-1/2,3/4 }{} \hspace{1.5em} \MatchingMeandersab{4}{-4/-2, -3/1,-1/3,2/4 }{} 
\hspace{1.5em} \MatchingMeandersab{4}{-4/-1, -3/2,-2/3,1/4 }{} 
\end{center}  
\begin{tikzpicture}[scale=0.5]
\node at (0,1) {};
\node at (4,1) {$ \scriptscriptstyle{\s\langle x_{\bar 4}, x_{\bar 3}\rangle \langle x_{\bar 2}, x_{ 1}\rangle  \langle x_{\bar 1}, x_{2}\rangle\langle x_{3}, x_{4}\rangle }$};
\node at (13.5,1) {$\scriptscriptstyle{\s \q \langle x_{\bar 4}, x_{\bar 2}\rangle \langle x_{\bar 3}, x_{1}\rangle  \langle x_{\bar 1}, x_{3}\rangle\langle x_{2}, x_{4}\rangle }$};
\node at (23,1) {$\scriptscriptstyle{\alpha\q^2 \langle x_{\bar 4}, x_{\bar 1}\rangle \langle x_{\bar 3}, x_{ 2}\rangle  \langle x_{1}, x_{4}\rangle\langle x_{\bar 2}, x_{3}\rangle }$};
\end{tikzpicture} 
\begin{center} 
\MatchingMeandersab{4}{-4/3, -2/-1,1/2,-3/4 }{} \hspace{1.5em} \MatchingMeandersab{4}{-4/2, -3/-1,1/3,-2/4 }{} \hspace{1.5em} \MatchingMeandersab{4}{-4/1, -3/-2,2/3,-1/4 }{} 
\end{center} 
\begin{tikzpicture}[scale=0.5]
\node at (0,1) {};
\node at (4,1) {$ \scriptscriptstyle{\s\langle x_{\bar 4}, x_{ 3}\rangle \langle x_{\bar 2}, x_{ \bar 1}\rangle  \langle x_{ 1}, x_{ 2}\rangle\langle x_{\bar 3}, x_{4}\rangle }$};
\node at (13.5,1) {$\scriptscriptstyle{\s\q \langle x_{\bar 4}, x_{ 2}\rangle \langle x_{\bar 3}, x_{\bar 1}\rangle  \langle x_{ 1}, x_{3}\rangle\langle x_{\bar 2}, x_{4}\rangle }$};
\node at (23,1) {$\scriptscriptstyle{\s\langle x_{\bar 4}, x_{ 1}\rangle \langle x_{\bar 3}, x_{ \bar 2}\rangle  \langle x_{\bar 1}, x_{4}\rangle\langle x_{ 2}, x_{3}\rangle }$};
\end{tikzpicture}  
\begin{center} 
\MatchingMeandersab{4}{-4/3, -2/1,-1/2,-3/4 }{} \hspace{1.5em} \MatchingMeandersab{4}{-4/2, -3/1,-1/3,-2/4 }{} 
\hspace{1.5em} \MatchingMeandersab{4}{-4/1, -3/2,-2/3,-1/4 }{} 
\end{center}
\begin{tikzpicture}[scale=0.5]
\node at (0,1) {};
\node at (4,1) {$ \scriptscriptstyle{\s^2\langle x_{\bar 4}, x_{ 3}\rangle \langle x_{\bar 2}, x_{  1}\rangle  \langle x_{ \bar 1}, x_{ 2}\rangle\langle x_{\bar 3}, x_{4}\rangle }$};
\node at (13.5,1) {$\scriptscriptstyle{\s^2\q \langle x_{\bar 4}, x_{ 2}\rangle \langle x_{\bar 3}, x_{1}\rangle  \langle x_{\bar  1}, x_{3}\rangle\langle x_{\bar 2}, x_{4}\rangle }$};
\node at (23,1) {$\scriptscriptstyle{\s^2q^2\langle x_{\bar 4}, x_{ 1}\rangle \langle x_{\bar 3}, x_{ 2}\rangle  \langle x_{\bar 1}, x_{4}\rangle\langle x_{ \bar 2}, x_{3}\rangle }$};
\end{tikzpicture}

\begin{center}
\MatchingMeandersab{4}{-4/-3, -3/-2, -2/-1, 1/2, 2/3, 3/4 }{} \hspace{1.5em} \MatchingMeandersab{4}{-4/-3, -3/-2, -2/1, -1/2, 2/3, 3/4}{} \hspace{1.5em} \MatchingMeandersab{4}{-4/-3, -3/2, 1/2, -2/-1, -2/3, 3/4}{} 
\end{center} 
\begin{tikzpicture}[scale=0.5]
\node at (0,1) {};
\node at (4,1) {$ \scriptscriptstyle{\langle x_{\bar 4}, T_{\bar 3} T_{\bar 2}x_{\bar 1} \rangle \langle x_{ 4}, T_{3} T_{2}x_{1} \rangle }$};
\node at (13.5,1) {$\scriptscriptstyle{\s \langle x_{4}, T_3 T_2 x_{\bar 1} \rangle \langle  x_{\bar 4}, T_{\bar 3} T_{\bar 2} x_{1} \rangle }$};
\node at (23,1) {$\scriptscriptstyle{\s \langle x_4, T_3 T_{\bar 2} x_{\bar 1} \rangle \langle x_{\bar 4}, T_{\bar 3} T_{2} x_{1} \rangle}$};
\end{tikzpicture}

\begin{center}
\MatchingMeandersab{4}{-4/3, 1/2, 2/3, -3/4, -2/-1, -3/-2}{} \hspace{1.5em} \MatchingMeandersab{4}{-4/-3, -3/2, -1/2, -2/1, -2/3, 3/4}{} \hspace{1.5em} \MatchingMeandersab{4}{-4/3, -2/3, -2/1, -1/2, -3/2, -3/4}{} 
\end{center} 

\begin{tikzpicture}[scale=0.5]
\node at (0,1) {};
\node at (4,1) {$ \scriptscriptstyle{\s \langle x_4, T_{\bar 3} T_{\bar 2}x_{\bar 1} \rangle \langle x_{\bar 4}, T_{3} T_{2}x_{1}\rangle }$};
\node at (13.5,1) {$\scriptscriptstyle{\s^2 \langle x_{\bar 4}, T_{\bar 3} T_2 x_{\bar 1} \rangle \langle x_{4}, T_{3} T_{\bar 2} x_{1} \rangle }$};
\node at (23,1) {$\scriptscriptstyle{\s^3 \langle x_4, T_{\bar 3} T_2 x_{\bar 1} \rangle \langle x_{\bar 4}, T_{3} T_{\bar 2} x_{1} \rangle }$};
\end{tikzpicture}

\begin{center}
\MatchingMeandersab{4}{-2/1, -3/-2, -3/4, -1/2, 2/3, -4/3}{} \hspace{1.5em} \MatchingMeandersab{4}{1/2, -3/2, -3/4, -2/-1, -2/3, -4/3}{} \hspace{1.5em} 
\end{center} 

\begin{tikzpicture}[scale=0.5]
\node at (0,1) {};
\node at (2,1) {$ \scriptscriptstyle{\s^2 \langle x_{\bar 4}, T_3 T_2 x_{\bar 1} \rangle \langle x_{4}, T_{\bar 3} T_{\bar 2} x_{1} \rangle  }$};
\node at (11,1) {$\scriptscriptstyle{\s^2 \langle x_{\bar 4}, T_3 T_{\bar 2} x_{\bar 1} \rangle \langle x_{4}, T_{\bar 3} T_{2} x_{1} \rangle }$};
\end{tikzpicture}
\caption{The statistics and the set of partitions of type B for the fourth moment. }
\label{figexample1}
\end{figure}

\end{example}
Let $\An$ be  the von Neumann algebra generated by $\{\G(x\otimes y)\mid x\otimes y \in \HH_\R\}$ for $|q|<1$.
\begin{proposition}
Suppose that $\dim(H_{\R}) \geq 2$, $q, \alpha \in (-1,1)$ and $T=\bar{T}=I$. 
Then from Example \ref{exkumulanty}  we can easily calculate   that the vacuum state is not a trace on $\An$. 
\end{proposition}

\begin{proof}[Proof of the last statement]
By using Example \ref{exkumulanty}, we obtain
\begin{align*}
\state(\G(x_{\overline{ 4}}\otimes x_{ 4})&\G(x_{\overline{ 3}}\otimes x_{ 3})\G(x_{\overline{ 2}}\otimes x_{ 2}) \G(x_{\bar 1}\otimes x_{ 1}))\\&=\langle x_{\bar 4}, x_{\bar 3}\rangle \langle x_{\bar 2}, x_{\bar 1}\rangle  \langle x_{1}, x_{2}\rangle\langle x_{3}, x_{4}\rangle +\q \langle x_{\bar 4}, x_{\bar 2}\rangle \langle x_{\bar 3}, x_{\bar 1}\rangle  \langle x_{1}, x_{3}\rangle\langle x_{2}, x_{4}\rangle \\&+ \langle x_{\bar 4}, x_{\bar 1}\rangle \langle x_{\bar 3}, x_{\bar 2}\rangle  \langle x_{1}, x_{4}\rangle\langle x_{2}, x_{3}\rangle +\s\langle x_{\bar 4}, x_{\bar 3}\rangle \langle x_{\bar 2}, x_{ 1}\rangle  \langle x_{\bar 1}, x_{2}\rangle\langle x_{3}, x_{4}\rangle \\&+\s \q \langle x_{\bar 4}, x_{\bar 2}\rangle \langle x_{\bar 3}, x_{1}\rangle  \langle x_{\bar 1}, x_{3}\rangle\langle x_{2}, x_{4}\rangle +\alpha\q^2 \langle x_{\bar 4}, x_{\bar 1}\rangle \langle x_{\bar 3}, x_{ 2}\rangle  \langle x_{1}, x_{4}\rangle\langle x_{\bar 2}, x_{3}\rangle 
\\&+ \s\langle x_{\bar 4}, x_{ 3}\rangle \langle x_{\bar 2}, x_{ \bar 1}\rangle  \langle x_{ 1}, x_{ 2}\rangle\langle x_{\bar 3}, x_{4}\rangle +\s\q \langle x_{\bar 4}, x_{ 2}\rangle \langle x_{\bar 3}, x_{\bar 1}\rangle  \langle x_{ 1}, x_{3}\rangle\langle x_{\bar 2}, x_{4}\rangle \\&+\s\langle x_{\bar 4}, x_{ 1}\rangle \langle x_{\bar 3}, x_{ \bar 2}\rangle  \langle x_{\bar 1}, x_{4}\rangle\langle x_{ 2}, x_{3}\rangle +\s^2\langle x_{\bar 4}, x_{ 3}\rangle \langle x_{\bar 2}, x_{  1}\rangle  \langle x_{ \bar 1}, x_{ 2}\rangle\langle x_{\bar 3}, x_{4}\rangle \\&+\s^2\q \langle x_{\bar 4}, x_{ 2}\rangle \langle x_{\bar 3}, x_{1}\rangle  \langle x_{\bar  1}, x_{3}\rangle\langle x_{\bar 2}, x_{4}\rangle +\s^2q^2\langle x_{\bar 4}, x_{ 1}\rangle \langle x_{\bar 3}, x_{ 2}\rangle  \langle x_{\bar 1}, x_{4}\rangle\langle x_{ \bar 2}, x_{3}\rangle \\&+\langle x_{\bar 4}, T_{\bar 3} T_{\bar 2}x_{\bar 1} \rangle \langle x_{ 4}, T_{3} T_{2}x_{1} \rangle + \s \langle x_{4}, T_3 T_2 x_{\bar 1} \rangle \langle \langle x_{\bar 4}, T_{\bar 3} T_{\bar 2} x_{1} \rangle \\&+\s \langle x_4, T_3 T_{\bar 2} x_{\bar 1} \rangle \langle x_{\bar 4}, T_{\bar 3} T_{2} x_{1} \rangle + \s \langle x_4, T_{\bar 3} T_{\bar 2}x_{\bar 1} \rangle \langle x_{\bar 4}, T_{3} T_{2}x_{1}\rangle \\&+ \s^2 \langle x_{\bar 4}, T_{\bar 3} T_2 x_{\bar 1} \rangle \langle x_{4}, T_{3} T_{\bar 2} x_{1} \rangle + \s^3 \langle x_4, T_{\bar 3} T_2 x_{\bar 1} \rangle \langle x_{\bar 4}, T_{3} T_{\bar 2} x_{1} \rangle \\&+ \s^2 \langle x_{\bar 4}, T_3 T_2 x_{\bar 1} \rangle \langle x_{4}, T_{\bar 3} T_{\bar 2} x_{1} \rangle   + \s^2 \langle x_{\bar 4}, T_3 T_{\bar 2} x_{\bar 1} \rangle \langle x_{4}, T_{\bar 3} T_{2} x_{1} \rangle 
\intertext{and by permuting $x_{\overline{ 4}}\otimes x_{ 4},\cdots, x_{\bar 1}\otimes x_{ 1}$}
\state(\G(x_{\overline{ 3}}\otimes x_{ 3})&\G(x_{\overline{ 2}}\otimes x_{ 2}) \G(x_{\bar 1}\otimes x_{ 1})\G(x_{\bar 4}\otimes x_{ 4}))\\&=\langle x_{\bar 3}, x_{\bar 2}\rangle \langle x_{\bar 1}, x_{\bar 4}\rangle  \langle x_{4}, x_{1}\rangle\langle x_{2}, x_{3}\rangle 
+\q \langle x_{\bar 3}, x_{\bar 1}\rangle \langle x_{\bar 2}, x_{\bar 4}\rangle  \langle x_{4}, x_{2}\rangle\langle x_{1}, x_{3}\rangle 
\\&+ \langle x_{\bar 3}, x_{\bar 4}\rangle \langle x_{\bar 2}, x_{\bar 1}\rangle  \langle x_{4}, x_{3}\rangle\langle x_{1}, x_{2}\rangle 
+\s\langle x_{\bar 3}, x_{\bar 2}\rangle \langle x_{\bar 1}, x_{ 4}\rangle  \langle x_{\bar 4}, x_{1}\rangle\langle x_{2}, x_{3}\rangle 
\\&+\s \q \langle x_{\bar 3}, x_{\bar 1}\rangle \langle x_{\bar 2}, x_{4}\rangle  \langle x_{\bar 4}, x_{2}\rangle\langle x_{1}, x_{3}\rangle
+\alpha\q^2 \langle x_{\bar 3}, x_{\bar 4}\rangle \langle x_{\bar 2}, x_{ 1}\rangle  \langle x_{4}, x_{3}\rangle\langle x_{\bar 1}, x_{2}\rangle 
\\&+ \s\langle x_{\bar 3}, x_{ 2}\rangle \langle x_{\bar 1}, x_{ \bar 4}\rangle  \langle x_{ 4}, x_{ 1}\rangle\langle x_{\bar 2}, x_{3}\rangle 
+\s\q \langle x_{\bar 3}, x_{ 1}\rangle \langle x_{\bar 2}, x_{\bar 4}\rangle  \langle x_{ 4}, x_{2}\rangle\langle x_{\bar 1}, x_{3}\rangle
\\&+\s\langle x_{\bar 3}, x_{ 4}\rangle \langle x_{\bar 2}, x_{ \bar 1}\rangle  \langle x_{\bar 4}, x_{3}\rangle\langle x_{ 1}, x_{2}\rangle 
+\s^2\langle x_{\bar 3}, x_{ 2}\rangle \langle x_{\bar 1}, x_{  4}\rangle  \langle x_{ \bar 4}, x_{ 1}\rangle\langle x_{\bar 2}, x_{3}\rangle 
\\&+\s^2\q \langle x_{\bar 3}, x_{ 1}\rangle \langle x_{\bar 2}, x_{4}\rangle  \langle x_{\bar  4}, x_{2}\rangle\langle x_{\bar 1}, x_{3}\rangle 
+\s^2q^2\langle x_{\bar 3}, x_{ 4}\rangle \langle x_{\bar 2}, x_{ 1}\rangle  \langle x_{\bar 4}, x_{3}\rangle\langle x_{ \bar 1}, x_{2}\rangle \\&+\langle x_{\bar 3}, T_{\bar 2} T_{\bar 1}x_{\bar 4} \rangle \langle x_{ 3}, T_{2} T_{1}x_{4} \rangle + \s \langle x_{3}, T_2 T_1 x_{\bar 4} \rangle \langle  x_{\bar 3}, T_{\bar 2} T_{\bar 1} x_{4} \rangle \\&+\s \langle x_3, T_2 T_{\bar 1} x_{\bar 4} \rangle \langle x_{\bar 3}, T_{\bar 2} T_{1} x_{4} \rangle + \s \langle x_3, T_{\bar 2} T_{\bar 1}x_{\bar 4} \rangle \langle x_{\bar 3}, T_{2} T_{1}x_{4}\rangle \\&+ \s^2 \langle x_{\bar 3}, T_{\bar 2} T_1 x_{\bar 4} \rangle \langle x_{3}, T_{2} T_{\bar 1} x_{4} \rangle + \s^3 \langle x_3, T_{\bar 2} T_1 x_{\bar 4} \rangle \langle x_{\bar 3}, T_{2} T_{\bar 1} x_{4} \rangle \\&+ \s^2 \langle x_{\bar 3}, T_2 T_1 x_{\bar 4} \rangle \langle x_{3}, T_{\bar 2} T_{\bar 1} x_{4} \rangle   + \s^2 \langle x_{\bar 3}, T_2 T_{\bar 1} x_{\bar 4} \rangle \langle x_{3}, T_{\bar 2} T_{1} x_{4} \rangle 
\end{align*}

Since $\dim(H_{\R}) \geq 2$, there are two orthogonal unit eigenvectors $e_1, e_2$  and we take $x_{\bar 1}=x_4=e_1$, $x_{1}=x_{\bar 4}=e_2$, $x_{\bar{2}}=x_3=e_1$ and $x_{2}=x_{\bar{3}}=e_2$. Let us assume that $\bar T = T=I$.  Hence
\begin{align*}
&\state(\G(x_{\overline{ 4}}\otimes x_{ 4})\G(x_{\overline{ 3}}\otimes x_{ 3})\G(x_{\overline{ 2}}\otimes x_{ 2}) \G(x_{\bar 1}\otimes x_{ 1}))\\&-   \state(\G(x_{\overline{ 3}}\otimes x_{ 3})\G(x_{\overline{ 2}}\otimes x_{ 2}) \G(x_{\bar 1}\otimes x_{ 1})\G(x_{\overline{ 4}}\otimes x_{ 4}))\\&=\s^2\q^2-\s^2+\langle e_2, e_1 \rangle \langle e_1, e_2 \rangle + \s \langle e_1, e_1 \rangle \langle e_2, e_2 \rangle \\&+\s \langle e_1, e_1 \rangle \langle e_2, e_2 \rangle + \s \langle e_1, e_1 \rangle  \langle e_2, e_2 \rangle + \s^2 \langle e_2, e_1 \rangle \langle e_1, e_2 \rangle \\&+ \s^3 \langle e_1, e_1 \rangle \langle e_2, e_2 \rangle + \s^2 \langle e_2, e_1 \rangle \langle e_1, e_2 \rangle +   \s^2 \langle e_2,  e_1 \rangle \langle e_1, e_2 \rangle \\&- \langle e_2, e_2 \rangle \langle e_1, e_1 \rangle - \s \langle e_1, e_2 \rangle  \langle e_2, e_1 \rangle -\s \langle e_1, e_2 \rangle \langle e_2, e_1 \rangle \\&- \s \langle e_1, e_2 \rangle \langle e_2, e_1\rangle - \s^2 \langle e_2, e_2 \rangle \langle e_1, e_1 \rangle - \s^3 \langle e_1, e_2 \rangle \langle e_2, e_1 \rangle \\&- \s^2 \langle e_2, e_2 \rangle \langle e_1, e_1 \rangle   - \s^2 \langle e_2,  e_2 \rangle \langle e_1, e_1 \rangle = \alpha^2q^2-4\alpha^2 +3\alpha+\alpha^3-1
\end{align*}
Therefore the vacuum state is not a trace when $\s, q \in (-1,1)$.
\end{proof}

\begin{remark}
In \cite{BE23}, Bożejko and Ejsmont calculated for the double Gaussian operator of type B $G_{\alpha,q}$, if $\dim(H_{\R}) \geq 2$, $|q|<1$ and $T=\bar{T}=\mathbf{0}$, then vacuum state is a trace on $\text{vN}(G_{\alpha,q}(\HH_{\R}))$ if and only if $\s=0$.
\end{remark}

For $\pi \in \NC_2^A(2 m)$ we say that  B-pair  is inner if it is covered by another B-pair. A B-pair  of a noncrossing partition of type A is outer if it is not inner. 
In accordance with this definition we can formulate the following corollary.  

\begin{corollary}\label{cor13} Assume that $x_{\bar i},x_i \in H_\R$ for $i=1,\dots,n$. 
\begin{enumerate}
\item For $\s=0$, we recover the $\q$-deformed formula for moments of $q$-L\'evy process  \cite{Ans04}: 
\begin{align*}
\state(  B_{0,\q}(x_{\overline{n}}\otimes  x_{n})\cdots B_{0,\q}(x_{\bar 1}\otimes x_1))
&=\sum_{\substack{ \pi \in \P^{B}(n)\\ \NB(\pi)=0 }} \q^{\Cr(\pi)}  R_{\pi}^{x ,T,0}
\\
&=\sum_{\pi \in \PA(n)} \q^{\Cr(\pi)} R_{\pi}^{x ,T,0} 
\end{align*}
\item  For $T=\mathbf{0}$ and $\lambda=0$ we get the formula for Gaussian operator of type B \cite{BE23} 
\begin{align*}
\state\big( {B}_{\alpha,q}(x_{{\bar n}}\otimes  x_{n})\cdots {B}_{\alpha,q}({ x_{\bar{1}}\otimes  x_{1}})\big)=\sum_{\pi \in \PB_{ 2}(n)} \alpha^{\NB(\pi)}q^{\rc(\pi)} \prod_{\substack{\{i,j\} \in \pi} }\langle x_i, x_j\rangle  
\end{align*}
\item For $\q=0$ and $x_{\bar i}=x_i$, we obtain the  formula  
\begin{align*}
\state(\G(x_{{n}}\otimes  x_{n})\cdots \G(x_{ 1}\otimes x_1))
= \sum_{\pi \in \NC^A(n)} \left(1+\alpha\right)^{\#\Out(\pi)}R_{\pi}^{x ,T,\lambda},
\end{align*}
where $\#\Out(\pi)$ is the number of outer arcs in a noncrossing partition of type A.
\end{enumerate}
\end{corollary}

\begin{proof}

(3). 
In order to proof the third point, we use the projector maps form \cref{jakliczyc} and by the assumption  $x_{\bar i }= x_i$,  we obtain  
\begin{align*}
\sum_{\pi \in \NC^A(n)} \left(1+\alpha\right)^{\#\Out(\pi)}R_{\pi}^{x ,T,\lambda} &=\sum_{\pi \in \NC^A(n)} \sum_{i=0}^{\#\Out(\pi)}\alpha^i{\#\Out(\pi) \choose i}R_{\pi}^{x ,T,\lambda}
\intertext{
 we count how many times we can assign parameter $\alpha$ to outer arcs of $\D^{-1}(\pi)$ and taking into account \cref{liczbapartycji} we see}
\\ &=\sum_{\pi \in \NC^A(n)}
\sum_{\sigma \in\D^{-1}(\pi)}\alpha^{\NB(\sigma)}  R_{\pi}^{x ,T,\lambda}
\\&=\sum_{\sigma \in \NC^B(n)} \s^{\NB(\pi)}R_{\pi}^{x ,T,\lambda}. 
\end{align*}

\end{proof}
\begin{remark}
The formula (3) in Corollary \ref{cor13} was employed in many papers related to Conditional Free Probability, e.g.  
\cite{BLS1996,Wojakowski2007,BozejkoWysoczanski2001}. 
\end{remark}

\begin{center} Acknowledgments
\end{center}
The authors would like to thank Marek Bo\.zejko and Franz Lehner for suggesting topics, several discussions and helpful comments. The authors also thank G. Świderski and R. Szwarc, who suggested and verified  the proof of point (5)    of \cref{uwagimiara}.  This research was funded in part by Narodowe Centrum Nauki, Poland WEAVE-UNISONO grant 2022/04/Y/ST1/00008 (Wiktor Ejsmont).

\bibliographystyle{amsalpha}

\bibliography{biblio2015}

\end{document}